\documentclass[10pt]{article}

\usepackage{latexsym, amssymb, amsmath, amsthm, amscd}
\usepackage[all,cmtip]{xy}
\usepackage{amsmath}
\usepackage{amsthm}
\usepackage{amssymb}
\usepackage{amsfonts}
\usepackage{amsrefs}
\usepackage{color,authblk}
\usepackage{tikz}
\usepackage{amscd} 
\usepackage{comment}
\usepackage{mathrsfs}
\usepackage{comment}
\usepackage[all]{xy}
\usepackage{graphicx}
\usepackage{authblk}
\usepackage{hyperref}
\usepackage{bbm}
\usepackage{fullpage}
\usepackage[all,cmtip]{xy}
\usepackage{relsize}
\usepackage{amsmath,amscd}
\usepackage{tikz-cd}
\usepackage{tikz}
\usetikzlibrary{matrix}
\usepackage[mathcal]{euscript}
\usepackage{txfonts}

\usepackage{hyperref}

\numberwithin{equation}{section} \DeclareMathSizes{2}{10}{12}{13}
\makeatletter
\newcommand*{\doublerightarrow}[2]{\mathrel{
  \settowidth{\@tempdima}{$\scriptstyle#1$}
  \settowidth{\@tempdimb}{$\scriptstyle#2$}
  \ifdim\@tempdimb>\@tempdima \@tempdima=\@tempdimb\fi
  \mathop{\vcenter{
    \offinterlineskip\ialign{\hbox to\dimexpr\@tempdima+1em{##}\cr
    \rightarrowfill\cr\noalign{\kern.5ex}
    \rightarrowfill\cr}}}\limits^{\!#1}_{\!#2}}}
\newcommand*{\triplerightarrow}[1]{\mathrel{
  \settowidth{\@tempdima}{$\scriptstyle#1$}
  \mathop{\vcenter{
    \offinterlineskip\ialign{\hbox to\dimexpr\@tempdima+1em{##}\cr
    \rightarrowfill\cr\noalign{\kern.5ex}
    \rightarrowfill\cr\noalign{\kern.5ex}
    \rightarrowfill\cr}}}\limits^{\!#1}}}
\makeatother

\newtheorem{thm}{Proposition}[section]
\newtheorem{Thm}[thm]{Theorem}

\newtheorem{cor}[thm]{Corollary}

\newtheorem{lem}[thm]{Lemma}
\newtheorem{defn}[thm]{Definition}

\title{Noncommutative supports, local cohomology and spectral  sequences}

\author{Abhishek Banerjee\footnote{Department of Mathematics, Indian Institute of Science, Bangalore, India. Email: abhishekbanerjee1313@gmail.com} $\qquad\qquad$ Surjeet Kour\footnote{Department of Mathematics, Indian Institute of Technology, Delhi, India. Email: koursurjeet@gmail.com} }

\date{}

\begin{document}

\maketitle 

\begin{abstract} The purpose of this paper is to study local cohomology in the noncommutative algebraic geometry framework of Artin and Zhang. The noncommutative spaces are obtained by base change of a Grothendieck category that is locally noetherian or strongly locally noetherian.   Using what we call elementary objects and their injective hulls, we develop a theory of supports and associated primes in these categories. We apply our theory to study a general functorial setup that requires certain conditions on the injective hulls of elementary objects and gives us spectral sequences for derived functors associated to local cohomology objects, as well as generalized local cohomology and also generalized Nagata ideal transforms.
\end{abstract}

\smallskip

{\bf MSC(2020) Subject Classification: 13D45, 16E30, 18E40, 18G10}  


\medskip

\section{Introduction}

If $k$ is a field and $A$ is a commutative $k$-algebra, the category $Mod-A$ of $A$-modules determines the geometry of the affine scheme $Spec(A)$. If $A$ is no longer commutative, the category $Mod-A$ of right $A$-modules plays the role of a noncommutative affine scheme. Suppose that $A$ is a commutative graded $k$-algebra and $Gr(A)$ is its category of graded modules. Then,  the geometry of the projective scheme $Proj(A)$  can be understood by means of the category $QGr(A)$ which is the quotient of $Gr(A)$ over torsion modules, i.e., modules which are obtained as filtered colimits of bounded graded modules (see \cite{AZ0}).  Accordingly, the category $QGr(A)$ plays the role of a noncommutative projective scheme in the theory of Artin and Zhang \cite{AZ0}, where $A$ is a noncommutative graded $k$-algebra.  In general, one can approach noncommutative algebraic geometry by adapting notions from commutative rings and the usual theory of schemes to arbitrary Grothendieck categories.

\smallskip
One of the notions that is fundamental in algebraic geometry is that of base change. Let $k$ be a field and let $A$, $R$ be $k$-algebras. Then, a module over the tensor product algebra $(A\otimes_kR)$ may be described as an ordinary $A$-module equipped with an additional $R$-module action. If we replace the category of $A$-modules by a general $k$-linear abelian category $\mathscr S$, we obtain the category $\mathscr S_R$ of ``$R$-module objects in $\mathscr S$.''  These module objects were introduced by Popescu \cite[p108]{Pop} and may be treated as modules over a noncommutative base change of $R$ by the abelian category $\mathscr S$. An $R$-module object in $\mathscr S$ consists of an object $\mathscr M\in \mathscr S$ along with a morphism of $k$-algebras
\begin{equation}
\rho_{\mathscr M}:R\longrightarrow \mathscr S(\mathscr M,\mathscr M)
\end{equation} In \cite{AZ}, Artin and Zhang systematically developed the properties of the abstract module category $\mathscr S_R$, where $\mathscr S$ is a Grothendieck category, and often a locally noetherian or strongly locally noetherian Grothendieck category. The theory developed by Artin and Zhang in \cite{AZ} includes tensor products, internal hom, localization, versions of Hilbert Basis Theorem, Nakayama Lemma, as well as derived functors $Tor$ and $Ext$. Several other fundamental constructions on $\mathscr S_R$ such as ind-objects, coherence and deformation theory, have been developed by Lowen and Van den Bergh \cite{LV}. In \cite{Banj}, we have also studied descent properties in $\mathscr S_R$. 

\smallskip
We note that a  number of properties of the category $\mathscr S$ extend to the module category $\mathscr S_R$, following the idea of noncommutative base change. For example, if $\mathscr S$ is locally finitely generated, so is $\mathscr S_R$. The version of Hilbert Basis Theorem proved by Artin and Zhang in \cite{AZ} shows that if $\mathscr S$ is locally noetherian and $R$ is a finitely generated $k$-algebra, then $\mathscr S_R$ is also locally noetherian. Since Grothendieck categories abound in nature, the theory of the abstract module category $\mathscr S_R$ is very general, and applies to a wide variety of situations. For instance, $\mathscr S$ could be the category of sheaves of abelian groups on a topological space, comodules over a coalgebra over a field, the modules over a ringed space, quasi-coherent sheaves over a scheme, or the category of comodules over a flat Hopf algebroid.

\smallskip In this paper, we develop a local cohomology theory in the noncommutative algebraic geometry of Artin and Zhang \cite{AZ0}, \cite{AZ}, by making use of abstract module categories. We let the $k$-linear Grothendieck category $\mathscr S=\mathscr S_k$ be strongly locally noetherian, whence it follows that $\mathscr S_R$ is locally noetherian for any noetherian commutative $k$-algebra $R$. We should note that our framework will give new results, even  in the case we take $\mathscr S$ to be the category of modules over a noncommutative algebra $A$ that is strongly locally noetherian, i.e., 
$A\otimes_kR$ is right noetherian for every commutative noetherian $k$-algebra $R$ (see \cite{ASZ}). If $C$ is a coalgebra over a field, its category of comodules is locally noetherian (see \cite[$\S$ 2.4.8]{Dasc} and \cite[$\S$ V.4.3]{Sten}). As with the theory in \cite{AZ}, we can  take $\mathscr S$ to be the category $Gr(A)$ of graded modules over a graded algebra $A$. The category $Gr(A)$ is strongly locally notherian whenever $A$ is strongly noetherian in a graded sense  (see \cite[$\S$ B5]{AZ}). We can also take $\mathscr S=QGr(A)$ for a noncommutative graded $k$-algebra $A$. These are the noncommutative projective schemes of Artin and Zhang 
\cite{AZ0}. If $A$ is strongly  noetherian in a graded sense (but not necessarily commutative), the category $QGr(A)$ is strongly locally noetherian and $QGr(A)_R=QGr(A\otimes_kR)$ (see \cite[$\S$ B8]{AZ}).

\smallskip The aims of our paper are threefold. First, we build a theory of associated primes and supports in the abstract module category $\mathscr S_R$ using what we call $R$-elementary objects in $\mathscr S_R$. If $\mathscr S$ was the category of modules over a commutative $k$-algebra $A$, then an $R$-elementary object in $\mathscr S_R$ behaves somewhat like a quotient $A/\mathfrak{q}$, where $\mathfrak q\subseteq A$ is a prime ideal. Accordingly, the injective hulls of these $R$-elementary objects in $\mathscr S_R$ play a key role in our theory.  

\smallskip
Our second purpose is to develop local cohomology in  $\mathscr S_R$. In fact, we work with a   more general framework, extending the local cohomology with respect to a pair of ideals $I$, $J\subseteq R$ introduced  by Takahashi, Yoshino and Yoshizawa \cite{TYY}. One of the key aspects of the theory in \cite{TYY} is that local cohomology is based on a nonclosed support $\mathbb W(I,J)$. In order to study the two variable local cohomology objects $H^\bullet_{I,J}(\_\_)$ in 
$\mathscr S_R$, we will need the support theory, associated primes and $R$-elementary objects in $\mathscr S_R$.  If $\mathscr M\in \mathscr S_R$ is a finitely generated object, we show that the set $Ass(\mathscr M)$ is finite. For an ideal $I\subseteq R$, the question of the finiteness of the set of associated primes of local cohomology objects is very interesting  in commutative algebra, and is also related to Faltings' 
local-global-principle for finiteness dimensions (see \cite{Falt}, and also \cite[$\S$ 9.6.2]{BS}). In fact, a number of questions related to the supports and associated primes of local cohomology have been studied extensively in the literature (see, for instance, \cite{Brd79}, \cite{BHell}, \cite{BLash}, \cite{BRS},   \cite{Hell}, \cite{Hun}, \cite{RS}). We give a condition for $Ass(H^i_{I,J}(\mathscr M)/\mathscr N)$ to be finite, where $\mathscr M\in \mathscr S_R$ is finitely generated and $\mathscr N\subseteq H^i_{I,J}(\mathscr M)$ is a finitely generated subobject. This generalizes a result of Brodmann and Faghani \cite{BLash} to the two variable local cohomology context, as well as to the abstract module category
$\mathscr S_R$. In particular, by taking $\mathscr N=0$, we have a condition for the collection of associated primes of $H^i_{I,J}(\mathscr M)$ to be finite.

\smallskip Finally, using properties of the injective hull of $R$-elementary objects, we apply our theory to describe a general functorial setup that is similar to the one recently introduced by \`Alvarez Montaner, Boix and Zarzuela in \cite{Alz} (see also \cite{Alz2}).  By considering systems of functors on $\mathscr S_R$ satisfying certain conditions with respect to their values on injective hulls of $R$-elementary objects in $\mathscr S_R$, we obtain a general spectral sequence for their derived functors in this context. In particular, this gives us a spectral sequence for two variable local cohomology objects in $\mathscr S_R$. In fact, it is striking that the nonclosed supports $\mathbb W(I,J)$ from \cite{TYY} appear naturally in the theory developed in \cite{Alz}.

\smallskip
We conclude by showing two further applications of the spectral sequence in this  formalism. The first is with generalized local cohomology functors given by derived functors of 
\begin{equation}\label{cr510gitr}
\Gamma_{V_I}:\mathscr S_R \longrightarrow \mathscr S_R \qquad \mathscr M\mapsto \underset{t\geq 1}{\varinjlim}\textrm{ }\underline{Hom}_R(V/I^tV,\mathscr M)
\end{equation} 
where $V$ is a finitely generated $R$-module and $\underline{Hom}(V,\_\_):  \mathscr S_R\longrightarrow \mathscr S_R$ is the internal hom of an $R$-module $V$ with respect to an object of $\mathscr S_R$.  The second is the ``generalized Nagata ideal transform'' (see \cite[$\S$ 4]{Alz} and \cite{DAS}) that we extend to $\mathscr S_R$ as follows
\begin{equation}\label{521gntitr}
\Delta_{V_I}:\mathscr S_R\longrightarrow \mathscr S_R\qquad \mathscr M\mapsto \underset{t\geq 1}{\varinjlim}\textrm{ }\underline{Hom}_R(I^tV,\mathscr M)
\end{equation} where $V$ is again a finitely generated $R$-module. In each case we obtain spectral sequences  for their derived functors.

\section{$R$-elementary objects, associated primes and injectives}

Let $k$ be a field. Throughout, we let  $\mathscr S=\mathscr S_k$ be a strongly locally noetherian $k$-linear category. We let $R$ be a commutative and noetherian $k$-algebra. Accordingly, the category $\mathscr S_R$ of $R$-module objects in $\mathscr S$ is locally noetherian. By \cite[$\S$ B3]{AZ}, there is a ``tensor product'' $\_\_\otimes_R\_\_:
\mathscr S_R\times R-Mod \longrightarrow \mathscr S_R$ that is right exact in both variables. If $V$ is a flat $R$-module, then the functor $\_\_\otimes_RV:\mathscr S_R\longrightarrow
\mathscr S_R$ is exact. Since $k$ is a field, the functor $\mathscr M\otimes_k \_\_:k-Mod\longrightarrow \mathscr S$ is exact for any $\mathscr M\in \mathscr S$. Consequently, 
$(\mathscr M\otimes_k R)\otimes_R \_\_: R-Mod\longrightarrow \mathscr S_R$ is exact for the $k$-algebra $R$ (see \cite[$\S$ C1]{AZ}). We will typically write $\otimes:=\otimes_k$. We also note that an $R$-module object in $\mathscr S$ may be described by giving $\mathscr M\in \mathscr S$ and a ``structure map'' $\mathscr M\otimes R\longrightarrow \mathscr M$ satisfying analogues of the usual associativity
and unit conditions (see \cite[$\S$ B3.14]{AZ}). For notions such as finitely generated objects, noetherian objects and locally noetherian categories that we will need throughout this paper, we refer the reader to \cite{AR}.

\smallskip
We will now introduce some notation. Let $r_{\mathscr M}$ denote the endomorphism induced on $\mathscr M\in \mathscr S_R$ by $r\in R$. We set the annihilator $Ann(\mathscr M)\subseteq R$ to be the ideal
\begin{equation}\label{3ann}
Ann(\mathscr M):=\{\mbox{$r\in R$ $\vert$ $r_{\mathscr M}=0:\mathscr M=\mathscr M\otimes_RR\xrightarrow{\mathscr M\otimes r}\mathscr M\otimes_RR=\mathscr M$}\}
\end{equation} It is clear that if $\mathscr M'\subseteq \mathscr M$, then $Ann(\mathscr M)\subseteq Ann(\mathscr M')$. 
For any $\mathscr M\in \mathscr S_R$ and any element $a\in R$, we form the short exact sequence
\begin{equation}\label{32ses}
0\longrightarrow Ker(a_{\mathscr M})\longrightarrow \mathscr M\longrightarrow a\mathscr M=Im(a_{\mathscr M})\longrightarrow 0
\end{equation}

\smallskip 
 For  any $\mathscr M\in \mathscr S_R$, we denote by $fg(\mathscr M)$ the collection of finitely generated subobjects of $\mathscr M$ in 
$\mathscr S_R$. We also let  $Pr(\mathscr M)$ denote the collection of finitely generated subobjects $0\ne \mathscr N\subseteq \mathscr M$ in $\mathscr S_R$ such that $Ann(\mathscr N)$ is a prime ideal in $R$. 

\begin{defn}\label{D3.1} We will say that an object $0\ne \mathscr L\in \mathscr S_R$ is $R$-elementary if it satisfies the following two conditions

\smallskip
(a) $\mathscr L$ is finitely generated.

\smallskip
(b) There is a prime ideal $\mathfrak p\subseteq R$ such that for any non-zero subobject $\mathscr L'\subseteq \mathscr L$, we have $Ann(\mathscr L')=\mathfrak p$.

\smallskip For $\mathscr M\in \mathscr S_R$, we will say that a prime ideal $\mathfrak p\subseteq R$ is an associated prime of $\mathscr M$ if there exists an $R$-elementary object 
$\mathscr L \subseteq \mathscr M$ such that
$Ann(\mathscr L)=\mathfrak p$. 
We denote by $Ass(\mathscr M)$ the collection of associated primes of $\mathscr M$.

\end{defn}

Our first purpose is to show that associated primes exist for every non-zero $\mathscr M\in \mathscr S_R$. For this, we first prove the following result.

\begin{lem}\label{L3.2} Let $0\ne \mathscr M\in \mathscr S_R$. Then, there exists a finitely generated subobject $0\ne \mathscr N\subseteq \mathscr M$ in $\mathscr S_R$ such that
$Ann(\mathscr N)$ is a prime ideal, i.e., $Pr(\mathscr M)$ is non-empty.
\end{lem}
\begin{proof}
Because $\mathscr S_R$ is locally noetherian, the object $0\ne \mathscr M\in \mathscr S_R$ must have non-zero finitely generated subobjects.  Since $R$ is noetherian, we can choose an ideal $I\subseteq R$ such that $I$ is maximal among annihilators of finitely generated non-zero subobjects of $\mathscr M$. Suppose that $I=Ann(\mathscr N)$ for some $0\ne \mathscr N\in fg(\mathscr M)$. We claim that $I$ is prime. Otherwise, suppose we have $a$, $b\in R$ such that $a\notin I$, $b\notin I$ but $ab\in I$. 

\smallskip
Since $ab\in I=Ann(\mathscr N)$, we have $a_{\mathscr N}\circ b_{\mathscr N}=0$, which implies that $b_{\mathscr N}:\mathscr N\longrightarrow \mathscr N$ factors through the kernel
$Ker(a_{\mathscr N})$. Hence, $\pi_{\mathscr N}^a\circ b_{\mathscr N}=0$, where $\pi_{\mathscr N}^a$ is the canonical epimorphism $\pi_{\mathscr N}^a:\mathscr N\longrightarrow a\mathscr N$. Since $\pi_{\mathscr N}^a$ is a morphism of $R$-module objects, we obtain  $b_{a\mathscr N}\circ \pi_{\mathscr N}^a=\pi_{\mathscr N}^a\circ b_{\mathscr N}=0$. Since $\pi_{\mathscr N}^a$ is an epimorphism, it follows that $b_{a\mathscr N}=0$, i.e., $b\in Ann(a\mathscr N)$. 

\smallskip
Since $\mathscr S_R$ is locally noetherian, we note that $a\mathscr N\subseteq \mathscr N$ is finitely generated. Since $a\notin I=Ann(\mathscr N)$, we know that $a\mathscr N\ne 0$. Clearly, $I=Ann(\mathscr N)\subseteq Ann(a\mathscr N)$. Then, $I\subsetneq (b,I)\subseteq Ann(a\mathscr N)$, which contradicts the maximality of $I$.
\end{proof}

\begin{thm}\label{P3.3}  Let $0\ne \mathscr M\in \mathscr S_R$. Then, $Ass(\mathscr M)\ne \phi$. 

\end{thm} 
\begin{proof}
Suppose that $Ass(\mathscr M)=\phi$. By Lemma \ref{L3.2}, we know that $Pr(\mathscr M)\ne \phi$. As such, we may choose $0\ne \mathscr N_0\in fg(\mathscr M)$ such that $\mathfrak p_0:=Ann(\mathscr N_0)$ is a prime ideal.  Since $\mathfrak p_0$ cannot be an associated prime of $\mathscr M$, we can choose $0\ne \mathscr N_0'\subseteq \mathscr N_0$ such that $\mathfrak p_0=Ann(\mathscr N_0)\subsetneq Ann(\mathscr N_0')$. Again by Lemma \ref{L3.2}, we know that $Pr(\mathscr N_0')\ne \phi$. Accordingly, we can choose $0\ne \mathscr N_1\subseteq \mathscr N_0'$ such that $\mathfrak p_1:=Ann(\mathscr N_1)$ is a prime ideal. But then,
\begin{equation}
\mathfrak p_0=Ann(\mathscr N_0)\subsetneq Ann(\mathscr N_0')\subseteq Ann(\mathscr N_1)=\mathfrak p_1
\end{equation} By repeating the argument, we obtain an infinite increasing chain of prime ideals in $R$. This contradicts the fact that $R$ is noetherian.
\end{proof}

\begin{thm}\label{P3.4}
Suppose that $0\longrightarrow \mathscr M'\longrightarrow\mathscr M\longrightarrow \mathscr M''\longrightarrow 0$ is a short exact sequence in $\mathscr S_R$. Then, we have 
$Ass(\mathscr M')\subseteq Ass(\mathscr M)\subseteq Ass(\mathscr M')\cup Ass(\mathscr M'')$. 
\end{thm}
\begin{proof}
From Definition \ref{D3.1}, it is clear that $Ass(\mathscr M')\subseteq Ass(\mathscr M)$. We now consider $\mathfrak p\in Ass(\mathscr M)$ and an $R$-elementary  object $ \mathscr L\subseteq \mathscr M$ such that
$Ann(\mathscr L)=\mathfrak p$. We set $\mathscr K:=\mathscr M'\cap \mathscr L$. 
Suppose that $\mathscr K=0$. Then, we have
\begin{equation}
\mathscr L=\mathscr L/\mathscr M'\cap \mathscr L\cong (\mathscr L+\mathscr M')/\mathscr M'\subseteq \mathscr M/\mathscr M'=\mathscr M''
\end{equation} Accordingly, $\mathfrak p\in Ass(\mathscr L)\subseteq Ass(\mathscr M'')$. Otherwise, suppose that $\mathscr K\ne 0$.  Then, $0\ne \mathscr K=\mathscr M'\cap \mathscr L\subseteq \mathscr L$ is also $R$-elementary with $\mathfrak p=Ann(\mathscr L)=Ann(\mathscr K)$. 
By Definition \ref{D3.1}, it is clear that $\mathfrak p\in Ass(\mathscr K)\subseteq Ass(\mathscr M')$. 
\end{proof}

\begin{thm}\label{P3.5}
Let $\mathscr M\in \mathscr S_R$ be finitely generated. Then, there exists a finite filtration
\begin{equation}\label{filt35}
0=\mathscr M_0 \subseteq \mathscr M_1 \subseteq \dots \subseteq \mathscr M_p=\mathscr M
\end{equation} such that each successive quotient $\mathscr M_i/\mathscr M_{i-1}$ is $R$-elementary. Additionally, $Ass(\mathscr M)\subseteq \underset{i=1}{\overset{p}{\bigcup}} Ass(\mathscr M_i/\mathscr M_{i-1})$.
\end{thm}

\begin{proof}
Let $\mathscr M\ne 0$. By Proposition \ref{P3.3}, we can choose a subobject $0\ne \mathscr M_1\subseteq \mathscr M$ that is $R$-elementary. If $\mathscr M=\mathscr M_1$, we are done. Otherwise, we continue by applying Proposition \ref{P3.3} to $\mathscr M/\mathscr M_1$, which gives us a filtration as in \eqref{filt35}. The filtration must be finite because $\mathscr S_R$ is a locally noetherian category and $\mathscr M\in \mathscr S_R$ is a finitely generated object. The last assertion follows by applying Proposition \ref{P3.4} repeatedly to the short exact sequences
$0\longrightarrow \mathscr M_{i-1}\longrightarrow \mathscr M_i\longrightarrow \mathscr M_i/\mathscr M_{i-1}\longrightarrow 0$. 
\end{proof}

\begin{cor}\label{C3.6}  Let $\mathscr M\in \mathscr S_R$ be finitely generated.  Then, the set $Ass(\mathscr M)$ is finite.
\end{cor}

\begin{proof}
This is clear from Proposition \ref{P3.5}. 
\end{proof}

\begin{cor}\label{C3.7} Let $\mathscr M\in \mathscr S_R$ and let $\mathscr N\subseteq \mathscr M$ be an essential subobject. Then, $Ass(\mathscr N)=Ass(\mathscr M)$.
\end{cor} 

\begin{proof}
It follows from Proposition \ref{P3.4} that $Ass(\mathscr N)\subseteq Ass(\mathscr M)$. On the other hand, let $\mathfrak p\in Ass(\mathscr M)$ and $0\ne \mathscr L\in fg(\mathscr M)$ be an $R$-elementary object such that $Ann(\mathscr L)=\mathfrak p$. Since $\mathscr N\subseteq \mathscr M$ is essential, we know $\mathscr N\cap \mathscr L\ne 0$. It is now clear from Definition \ref{D3.1} that $\mathscr N\cap \mathscr L$ is $R$-elementary and $Ann(\mathscr N\cap \mathscr L)=\mathfrak p$. 
\end{proof}

Since $\mathscr S_R$ is a Grothendieck category, every object $\mathscr M\in \mathscr S_R$ has an injective hull, which we will denote by $\mathscr E(\mathscr M)$. We will now prove the main result of this section, which describes the injectives of $\mathscr S_R$ in terms of injective envelopes of $R$-elementary objects. 

\begin{Thm}\label{T3.8bi}
Let $\mathscr S$ be a strongly locally noetherian Grothendieck category and let $R$ be a commutative noetherian $k$-algebra. Then, every injective object $\mathscr E$ in $\mathscr S_R$ can be expressed as a direct sum of injective hulls of $R$-elementary objects.
\end{Thm}

\begin{proof}
We choose a  family $\{\mathscr E_i\}_{i\in I}$ of subobjects of $\mathscr E$  satisfying the following two conditions

\smallskip
(1) Each $\mathscr E_i$ is isomorphic to the injective hull of an $R$-elementary object in $\mathscr S_R$

\smallskip
(2) The sum $\sum_{i\in I}\mathscr E_i\subseteq \mathscr E$ is direct, i.e., $(\mathscr E_i)\cap (\sum_{j\ne i}\mathscr E_j)=0$ for each $i\in I$

\smallskip
Since finite limits commute with filtered colimits in $\mathscr S_R$, condition (2) is equivalent to  $0=(\mathscr E_i)\cap (\sum_{j\in J}\mathscr E_j)$ for each $i\in I$ and any finite $J\subseteq I\backslash
\{i\}$. Accordingly, by Zorn's lemma, we may choose a maximal such family $\{\mathscr E_i\}_{i\in I_0}$ inside $\mathscr E$.

\smallskip
We now set $\mathscr E':=\bigoplus_{i\in I_0}\mathscr E_i$. Since $\mathscr S_R$ is locally noetherian, the direct sum $\mathscr E'$ of injectives must be injective. Accordingly, we can split $\mathscr E$ as $\mathscr E=\mathscr E'\oplus \mathscr E''$. If $\mathscr E''=0$, the result is already proved. If $\mathscr E''\ne 0$, it follows from Proposition \ref{P3.3} that there is an 
$R$-elementary object $\mathscr L\subseteq \mathscr E''$. Since $\mathscr E''$ is injective, the injective hull $\mathscr E(\mathscr L)\subseteq \mathscr E''$. This contradicts the maximality of the family
$\{\mathscr E_i\}_{i\in I_0}$. 

\end{proof}

If $T\subseteq R$ is a multiplicatively closed set, the localization $\mathscr M_T$ of an object $\mathscr M\in \mathscr S_R$ with respect to $T$ is given by
$\mathscr M\otimes_R R[T^{-1}]$ (see \cite[$\S$ B6]{AZ}). The localization $\mathscr M_T$ may also be expressed as the filtered colimit of copies of $\mathscr M$ connected by morphisms $t_{\mathscr M}:\mathscr M\longrightarrow \mathscr M$ over all $t\in T$  (see \cite[$\S$ B6.1]{AZ}). Accordingly, the kernel of the canonical map $\mathscr M\longrightarrow \mathscr M_T$ is given by the filtered colimit (see \cite[Proposition B6.2]{AZ})
\begin{equation}\label{kerun}
Ker(\mathscr M\longrightarrow \mathscr M_T)=\underset{t\in T}{\bigcup}\textrm{ }Ker(t_{\mathscr M}:\mathscr M\longrightarrow \mathscr M)
\end{equation}  Moreover, since $R[T^{-1}]$ is a flat $R$-module, it follows from \cite[Proposition C1.7]{AZ} that localization is an exact functor on $\mathscr S_R$.  If $\mathfrak p$ is a prime ideal, we denote by $\mathscr M_{\mathfrak p}$ the localization with respect
to $R\backslash {\mathfrak p}$. We now have the following definition.

\begin{defn}\label{D3.7l} Let $\mathscr M\in \mathscr S_R$. Then, the support $Supp(\mathscr M)$ of $\mathscr M$ consists of all prime ideals $\mathfrak p\subseteq R$ such that $\mathscr M_{\mathfrak p}\ne 0$.
\end{defn}

\begin{thm}\label{L3.8uh}
(a) Let $\mathscr M\in \mathscr S_R$ and let $\mathfrak p\in Supp(\mathscr M)$. Then, $\mathfrak p\supseteq Ann(\mathscr M)$. 

\smallskip
(b) If $\mathscr M\in \mathscr S_R$ is finitely generated, then $Supp(\mathscr M)=\mathbb V(Ann(\mathscr M))$, the collection of prime ideals containing $Ann(\mathscr M)$.

\smallskip
(d) For a short exact sequence $0\longrightarrow \mathscr M'\longrightarrow \mathscr M\longrightarrow \mathscr M''
\longrightarrow 0$ in $\mathscr S_R$, we have $Supp(\mathscr M)=Supp(\mathscr M')\cup Supp(\mathscr M'')$.
\end{thm}

\begin{proof}
(a) If $\mathfrak p\not\supseteq Ann(\mathscr M)$, we choose $t\in Ann(\mathscr M)$ such that $t\notin \mathfrak p$. Since $t$ is a unit in $R_{\mathfrak p}$, it follows that $\mathscr M\otimes_Rt=t_{\mathscr M}\otimes_RR_{\mathfrak p}=0:\mathscr M_{\mathfrak p}=\mathscr M\otimes_RR_{\mathfrak p}\longrightarrow \mathscr M\otimes_RR_{\mathfrak p}=\mathscr M_{\mathfrak p}$ is an isomorphism. This gives $\mathscr M_{\mathfrak p}=0$, and hence $\mathfrak p\notin Supp(\mathscr M)$. This proves (a).

\smallskip
Additionally, suppose that $\mathscr M$ is finitely generated and ${\mathfrak p}\in \mathbb  V(Ann(\mathscr M))$, but $\mathscr M_{\mathfrak p}=0$. Since the colimit in \eqref{kerun} is filtered, this means that $Ker(\mathscr M\longrightarrow\mathscr M_{\mathfrak p})=\mathscr M=Ker(t_{\mathscr M})$ for some $t\notin {\mathfrak p}$. Then, $t\in Ann(\mathscr M)$ and $t\notin {\mathfrak p}$, which is a contradiction. This proves (b). 
The result of (c) follows directly from the fact that localization is exact. 

\end{proof}

We conclude this section by establishing some properties of injective hulls in $\mathscr S_R$ that we will need throughout this paper. We begin with the following result.

\begin{lem}\label{L2.11fvd}
Let $\mathscr M\in \mathscr S_R$ and let $T\subseteq R$ be a multiplicatively closed subset. Then, any subobject of  $\mathscr M_T\in \mathscr S_{R[T^{-1}]}$ is of the form $\mathscr N_T$ for some subobject $\mathscr N\subseteq \mathscr M$ in $\mathscr S_R$. 
\end{lem}

\begin{proof}
Let $\mathscr K\subseteq \mathscr M_T\in \mathscr S_{R[T^{-1}]}$ be a subobject. We set up the following two pullback diagrams
\begin{equation}\label{2.7cdu}
\begin{array}{lll}
\begin{CD}
\mathscr N @>>> \mathscr M \\
@VVV @VVV \\
\mathscr K @>>> \mathscr M_T\\
\end{CD}
&\qquad\qquad\Rightarrow \qquad \qquad&\begin{CD}
\mathscr N \otimes_RR[T^{-1}]@>>> \mathscr M\otimes_RR[T^{-1}] =\mathscr M_T\\
@VVV @VVV \\
\mathscr K=\mathscr K\otimes_RR[T^{-1}] @>>> \mathscr M_T\otimes_RR[T^{-1}]=\mathscr M_T\\
\end{CD}\\
\end{array}
\end{equation} where the right hand square in \eqref{2.7cdu} follows by applying the exact functor $\_\_\otimes_RR[T^{-1}]$ to the left hand square.  
We note that $\mathscr M_T=\mathscr M\otimes_RR[T^{-1}]=\mathscr M\otimes_RR[T^{-1}]\otimes_RR[T^{-1}]=\mathscr M_T\otimes_RR[T^{-1}]$ and 
$\mathscr K\otimes_RR[T^{-1}] =\mathscr K\otimes_{R[T^{-1}]}R[T^{-1}]\otimes_RR[T^{-1}]=\mathscr K$. From \eqref{2.7cdu}, it is now clear that $\mathscr N$ is a subobject of
$\mathscr M$ and that $\mathscr K=\mathscr N \otimes_RR[T^{-1}]$. 
\end{proof}

\begin{thm}\label{P2.12qa}
Let $\mathscr E\in \mathscr S_R$ be an injective object and let $T\subseteq R$ be a multiplicatively closed subset. Then, the localization $\mathscr E_T$ is an injective object in
$\mathscr S_{R[T^{-1}]}$. 
\end{thm}
\begin{proof}
Let $\{\mathscr P_i\}_{i\in I}$ be a set of noetherian generators for $\mathscr S_R$. Since $\mathscr S$  is strongly locally noetherian, it follows from \cite[Proposition B5.1, Corollary B3.17]{AZ} that $\{\mathscr P_i\otimes_RR[T^{-1}]\}_{i\in I}$ is a set of noetherian generators for $\mathscr S_{R[T^{-1}]}$. We choose one of these generators $\mathscr P\in \{\mathscr P_i\}_{i\in I}$ and consider a subobject $\mathscr K\subseteq \mathscr P_T$ in $\mathscr S_{R[T^{-1}]}$. Applying Lemma \ref{L2.11fvd}, we have a subobject $\mathscr N\hookrightarrow 
\mathscr P$ in $\mathscr S_R$ such that the localization $\mathscr N_T=\mathscr K$. We note that since $\mathscr P\in \mathscr S_R$ is noetherian, so is $\mathscr N$. 

\smallskip
Since $\mathscr E\in \mathscr S_R$ is injective, it follows that we have a surjection $\mathscr S_R(\mathscr P,\mathscr E)\twoheadrightarrow \mathscr S_R(
\mathscr N,\mathscr E)$.  As mentioned above, the localization $\mathscr E_T$ may be expressed as a filtered colimit of copies of $\mathscr E$. Since $\mathscr N$, $\mathscr P$ are finitely generated, we have induced surjections $\mathscr S_R(\mathscr P,\mathscr E_T)\twoheadrightarrow \mathscr S_R(
\mathscr N,\mathscr E_T)$. In other words, we have a surjection $\mathscr S_{R[T^{-1}]}(
\mathscr P_T,\mathscr E_T)=\mathscr S_{R[T^{-1}]}(\mathscr P\otimes_R R[T^{-1}],\mathscr E_T)\twoheadrightarrow \mathscr S_{R[T^{-1}]}(
\mathscr N\otimes_RR[T^{-1}],\mathscr E_T)=\mathscr S_{R[T^{-1}]}(
\mathscr K,\mathscr E_T)$. Applying Baer's criterion  (see  \cite[Prop V.2.9]{Sten}) to the Grothendieck category $\mathscr S_{R[T^{-1}]}$, it follows that $\mathscr E_T$ is injective.
\end{proof}

We now need the following fact, which does not seem to have appeared before in the literature.

\begin{lem}\label{L2.13rdv}
Let $\mathscr D$ be a locally noetherian category. Then, a filtered colimit of essential monomorphisms in $\mathscr D$ is an essential monomorphism.
\end{lem}

\begin{proof}
Let $\{\mathscr B_i\longrightarrow \mathscr A_i\}_{i\in I}$ be a filtered system of essential monomorphisms. Then, $\mathscr B:=\underset{i\in I}{\varinjlim}\textrm{ }\mathscr B_i\longrightarrow \mathscr A:=\underset{i\in I}{\varinjlim}\textrm{ }\mathscr A_i$ is a monomorphism and we claim that $\mathscr B\subseteq \mathscr A$ is essential. For this, we consider a subobject $0\ne \mathscr F\subseteq \mathscr A$ and suppose that $\mathscr F\cap \mathscr B=0$. Since every object in $\mathscr D$ is the sum of its finitely generated subobjects, it is enough to consider the case where $\mathscr F$ is finitely generated. 

\smallskip
Accordingly, we can choose $j\in I$ large enough such that the inclusion $\mathscr F\hookrightarrow \mathscr A$ factors through the canonical morphism $\mathscr A_j\longrightarrow \mathscr A$.  Then, $\mathscr F\longrightarrow \mathscr A_j$ must be a monomorphism.  Since $\mathscr B_j\subseteq \mathscr A_j$ is essential, we have $\mathscr F\cap \mathscr B_j\ne 0$. Since $\mathscr F\cap \mathscr B_j\subseteq \mathscr F\cap \mathscr B$, it follows that  $\mathscr F\cap \mathscr B\ne 0$, which is a contradiction.
\end{proof}

\begin{thm}\label{P2.14ju} Let $T\subseteq R$ be a multiplicatively closed subset. Let $\mathscr N\in \mathscr S_R$ and let $\mathscr E(\mathscr N)$ be its injective hull. Then, the localization $\mathscr E(\mathscr N)_T$ is the injective hull of $\mathscr N_T$ in $\mathscr S_{R[T^{-1}]}$. 

\end{thm}
\begin{proof}
We know that the canonical inclusion $\mathscr N\hookrightarrow\mathscr E(\mathscr N)$ is essential. Since the localization is obtained by taking a filtered colimit, it follows from 
Lemma \ref{L2.13rdv} that $\mathscr N_T\hookrightarrow \mathscr E(\mathscr N)_T$ is essential. Since $\mathscr E(\mathscr N)$ is injective in $\mathscr S_R$, we know from Proposition \ref{P2.12qa} that $\mathscr E(\mathscr N)_T$ is injective in $\mathscr S_{R[T^{-1}]}$. The result is now clear. 
\end{proof}

\begin{thm}\label{P2.15cwi}
Let $\mathscr S$ be a strongly locally noetherian Grothendieck category and let $R$ be a commutative noetherian $k$-algebra. Let $\mathscr N\subseteq \mathscr M$ be an essential subobject in $\mathscr S_R$. Then, $Supp(\mathscr N)=Supp(\mathscr M)$. 
\end{thm}

\begin{proof}
From Proposition \ref{P2.14ju}, it is clear that $Supp(\mathscr N)=Supp(\mathscr E(\mathscr N))$. In general, if $\mathscr N\subseteq \mathscr M$ is essential, there is an inclusion
$\mathscr M\hookrightarrow \mathscr E(\mathscr N)$. Accordingly, $Supp(\mathscr N)=Supp(\mathscr M)=Supp(\mathscr E(\mathscr N))$. 
\end{proof}

\section{Local cohomology objects for a pair of ideals}

Let $I$, $J$ be ideals in $R$.  If $\mathscr M\in \mathscr S_R$ and $\mathfrak G$ is any family of subobjects of $\mathscr M$, we write $\bigoplus\mathfrak G$ (resp. $\sum \mathfrak G$) for the direct sum (resp.  the sum in $\mathscr M$) of all the objects in $\mathfrak G$. We now define
\begin{equation}\label{4.1dn}
|\Gamma_{I,J}(\mathscr M)|:=\{\mbox{$\mathscr N\in fg(\mathscr M)$ $\vert$ $I^n\subseteq Ann(\mathscr N)+J$ for $n\gg 1$ }\}\qquad
\Gamma_{I,J}(\mathscr M):=\sum |\Gamma_{I,J}(\mathscr M)| 
\end{equation} The definition in \eqref{4.1dn} is motivated by the notion of local cohomology for modules with respect to a pair of ideals introduced by
Takahashi, Yoshino and Yoshizawa in \cite{TYY}.  We will say that $\mathscr M\in \mathscr S_R$ is $(I,J)$-torsion if $\Gamma_{I,J}(\mathscr M)=\mathscr M$. We will say that $\mathscr M
\in \mathscr S_R$ is $(I,J)$-torsion free if $\Gamma_{I,J}(\mathscr M)=0$. When $J=0$, we will denote $\Gamma_{I,J}$ simply by $\Gamma_I$. 

\begin{lem}\label{L4.1qu} (a) 
For $\mathscr M\in \mathscr S_R$, the family $|\Gamma_{I,J}(\mathscr M)|$ is closed under subobjects and finite sums. Further, any finitely generated  $\mathscr N \subseteq \Gamma_{I,J}(\mathscr M)$ lies in the family $|\Gamma_{I,J}(\mathscr M)|$.

\smallskip
(b) $\Gamma_{I,J}$ is a left exact functor on $\mathscr S_R$.
\end{lem}

\begin{proof} (a)  It is clear that $|\Gamma_{I,J}(\mathscr M)|$ is closed under subobjects. We consider $\mathscr N_1$, $\mathscr N_2\in |\Gamma_{I,J}(\mathscr M)|$, with $I^{n_1}\subseteq Ann(\mathscr N_1)+J$ and $I^{n_2}\subseteq Ann(\mathscr N_2)+J$. Then, we have
\begin{equation}
I^{n_1+n_2}\subseteq (Ann(\mathscr N_1)+J)(Ann(\mathscr N_2)+J)\subseteq Ann(\mathscr N_1)\cap Ann(\mathscr N_2)+J\subseteq Ann(\mathscr N_1+\mathscr N_2)+J
\end{equation} Hence, $|\Gamma_{I,J}(\mathscr M)|$ is closed under finite sums.  Now if $\mathscr N \subseteq \Gamma_{I,J}(\mathscr M)$ is finitely generated, it follows from the definition in \eqref{4.1dn} that $\mathscr N$ lies within a finite sum of objects in $|\Gamma_{I,J}(\mathscr M)|$. Since $|\Gamma_{I,J}(\mathscr M)|$ is closed under subobjects and finite sums, the result is now clear.

\smallskip
(b) Let $\phi:\mathscr M'\longrightarrow \mathscr M$ be a morphism in $\mathscr S_R$. For any $\mathscr N'\subseteq \mathscr M'$, the quotient 
$\phi(\mathscr N')\subseteq \mathscr M$ satisfies $Ann(\mathscr N')\subseteq Ann(\phi(\mathscr N'))$. Accordingly, if $\mathscr N'\in |\Gamma_{I,J}(\mathscr M')|$, then $\phi(\mathscr N')\in |\Gamma_{I,J}(\mathscr M)|$.  This determines a morphism $\Gamma_{I,J}(\phi):\Gamma_{I,J}(\mathscr M')\longrightarrow 
\Gamma_{I,J}(\mathscr M)$. In particular, if $\phi:\mathscr M'\longrightarrow \mathscr M$ is a monomorphism, we have $|\Gamma_{I,J}(\mathscr M')|\subseteq 
|\Gamma_{I,J}(\mathscr M)|$. Hence, $\Gamma_{I,J}(\mathscr M')\subseteq \Gamma_{I,J}(\mathscr M)$. 
\end{proof}

As defined in \cite{TYY}, we now consider for the pair of ideals $(I,J)$ the set
\begin{equation}\label{wij}
\mathbb W(I,J)=\{\mbox{$\mathfrak p\in Spec(R)$ $\vert$ $I^n\subseteq \mathfrak p+J$ for $n\gg 1$}\}
\end{equation} When $J=0$, we note that $\mathbb W(I,J)=\mathbb V(I)$, the set of prime ideals containing $I$.

\begin{lem}\label{L4.2wo}
Let $\mathscr N\in \mathscr S_R$ be finitely generated. Then, $Supp(\mathscr N)\subseteq \mathbb W(I,J)$ if and only if $I^n\subseteq Ann(\mathscr N)+J$ for 
$n\gg 1$. 
\end{lem}

\begin{proof} 
Since $\mathscr N$ is finitely generated, it follows by Proposition \ref{L3.8uh} that $Supp(\mathscr N)=\mathbb V(Ann(\mathscr N))$. First, we suppose that $I^n\subseteq Ann(\mathscr N)+J$. Then, if $\mathfrak p\in Supp(\mathscr N)$, i.e., $\mathfrak p\supseteq Ann(\mathscr N)$, we see that $I^n\subseteq J+\mathfrak p$, i.e., $\mathfrak p\in \mathbb W(I,J)$.

\smallskip
Conversely, suppose that $Supp(\mathscr N)\subseteq \mathbb W(I,J)$. If $\mathfrak p$ is a prime ideal containing $Ann(\mathscr N)$, then there is $n_{\mathfrak p}\geq 1$ such that $I^{n_{\mathfrak p}}\subseteq {\mathfrak p}+J$. Since $R$ is noetherian, there are only finitely many primes mininal over $Ann(\mathscr N)$. It follows that we can find $n'\geq 1$ such that
$I^{n'}\subseteq rad(Ann(\mathscr N))+J$. Again since $R$ is noetherian, some power of $rad(Ann(\mathscr N))$ lies in $Ann(\mathscr N)$ and the result follows. 
\end{proof}

It follows from Lemma \ref{L4.2wo} that $\Gamma_{I,J}(\mathscr M)$ is the sum of all finitely generated subobjects $\mathscr N\in fg(\mathscr M)$ satisfying 
$Supp(\mathscr N)\subseteq \mathbb W(I,J)$. We note that this implies $Supp(\Gamma_{I,J}(\mathscr M))\subseteq \mathbb W(I,J)$. We will now describe the associated primes of $\Gamma_{I,J}(\mathscr M)$. 

\begin{thm}\label{P4.3ed}
For $\mathscr M\in \mathscr S_R$, we have $Ass(\mathscr M)\cap \mathbb W(I,J)=Ass(\Gamma_{I,J}(\mathscr M))$. 
\end{thm}
\begin{proof}
Since $\Gamma_{I,J}(\mathscr M)\subseteq \mathscr M$, we have $Ass(\Gamma_{I,J}(\mathscr M))\subseteq Ass(\mathscr M)$. If ${\mathfrak p}\in Ass(\Gamma_{I,J}(\mathscr M))$, we consider an $R$-elementary object $\mathscr L\subseteq  \Gamma_{I,J}(\mathscr M)$ such that $Ann(\mathscr L)={\mathfrak p}$. By Lemma \ref{L4.1qu} and Lemma \ref{L4.2wo}, we see that
${\mathfrak p}\in Ass(\mathscr L)\subseteq Supp(\mathscr L)\subseteq \mathbb W(I,J)$. Hence, $Ass(\Gamma_{I,J}(\mathscr M))\subseteq Ass(\mathscr M)\cap \mathbb W(I,J)$.

\smallskip
Conversely, we take ${\mathfrak p}\in Ass(\mathscr M)\cap \mathbb W(I,J)$ and consider an $R$-elementary object $\mathscr L\subseteq \mathscr M$ such that $Ann(\mathscr L)={\mathfrak p}$. Since ${\mathfrak p}\in \mathbb W(I,J)$, we can take $n\geq 1$ such that $I^n\subseteq {\mathfrak p}+J=Ann(\mathscr L)+J$. Hence, $\mathscr L\in |\Gamma_{I,J}(\mathscr M)|$ and therefore $\mathscr L\subseteq \Gamma_{I,J}(\mathscr M)$. This gives ${\mathfrak p}\in Ass(\Gamma_{I,J}(\mathscr M))$. 
\end{proof}

If $\mathcal A$ is an abelian category, we recall (see, for instance, \cite[$\S$ 1.1]{BeRe}) that a torsion theory on $\mathcal A$ consists of a pair $(\mathcal T,\mathcal F)$ of strict and full subcategories such that $\mathcal A(X,Y)=0$ for any $X\in \mathcal T$, $Y\in \mathcal F$ and any $Z\in \mathcal A$ fits into a short exact sequence $0\longrightarrow Z^{\mathcal T}
\longrightarrow Z\longrightarrow Z^{\mathcal F}\longrightarrow 0$ with $Z^{\mathcal T}\in \mathcal T$ and $Z^{\mathcal F}\in \mathcal F$.  Additionally, $(\mathcal T,\mathcal F)$ is said to be hereditary if the torsion class $\mathcal T$ is closed under subobjects.

\begin{thm}\label{torth} For ideals $I$, $J$, the pair $(\mathscr T(\mathscr S)_{I,J},\mathscr F(\mathscr S)_{I,J})$ of full subcategories of $\mathscr S_R$, where 
\begin{equation}
Ob(\mathscr T(\mathscr S)_{I,J}):=\{\mbox{$\mathscr M\in \mathscr S_R$ $\vert$ $\Gamma_{I,J}(\mathscr M)=\mathscr M$}\}
\qquad
Ob(\mathscr F(\mathscr S)_{I,J}):=\{\mbox{$\mathscr M\in \mathscr S_R$ $\vert$ $\Gamma_{I,J}(\mathscr M)=0$}\}
\end{equation}
forms a hereditary torsion theory in $\mathscr S_R$.
\end{thm}

\begin{proof}
We consider $\phi:\mathscr M'\longrightarrow \mathscr M''$ in $\mathscr S_R$, where $\mathscr M'\in \mathscr T(\mathscr S)_{I,J}$, $\mathscr M''\in \mathscr F(\mathscr S)_{I,J}$. We consider $\mathscr N'\in fg(\mathscr M')$.  Then, $Supp(\phi(\mathscr N'))\subseteq 
Supp(\mathscr N')\subseteq \mathbb W(I,J)$ and hence $\phi(\mathscr M')\subseteq \Gamma_{I,J}(\mathscr M'')=0$. Now for any $\mathscr M\in \mathscr S_R$, we consider the short exact sequence
\begin{equation}
0\longrightarrow \Gamma_{I,J}(\mathscr M)\longrightarrow \mathscr M\longrightarrow \mathscr M/\Gamma_{I,J}(\mathscr M)\longrightarrow 0
\end{equation} By Lemma \ref{L4.1qu}, it is clear that $\Gamma_{I,J}(\mathscr M)\in \mathscr T(\mathscr S)_{I,J}$. If $\mathscr M/\Gamma_{I,J}(\mathscr M)\notin\mathscr F(\mathscr S)_{I,J}$, there is some $0\ne \mathscr N\in fg(\mathscr M/\Gamma_{I,J}(\mathscr M))$ such that $Supp(\mathscr N)\subseteq \mathbb W(I,J)$. Then, we have a short exact sequence
\begin{equation}\label{4.6ses}
0\longrightarrow   \Gamma_{I,J}(\mathscr M)\longrightarrow \mathscr N_0\longrightarrow \mathscr N\longrightarrow 0
\end{equation} where $\mathscr N_0$ is the preimage of $\mathscr N$ in $\mathscr M$. Then, $Supp(\mathscr N_0)=Supp( \Gamma_{I,J}(\mathscr M))\cup Supp(\mathscr N)
\subseteq \mathbb W(I,J)$ and hence $\mathscr N_0\subseteq \Gamma_{I,J}(\mathscr M)$, which implies that $\mathscr N=0$. Hence, $(\mathscr T(\mathscr S)_{I,J},\mathscr F(\mathscr S)_{I,J})$ is a torsion theory on $\mathscr S_R$. The fact that $\mathscr T(\mathscr S)_{I,J}$ is closed under subobjects is also clear from Lemma \ref{L4.2wo}. 
\end{proof}

\begin{defn}\label{D4.4fv}
Let $\mathscr M\in \mathscr S_R$ and let $I,J\subseteq R$ be ideals. Then, for any $i\geq 0$, the $i$-th local cohomology object $H^i_{I,J}(\mathscr M)$ of $\mathscr M\in \mathscr S_R$ with respect to $(I,J)$ is given by the $i$-th right derived functor of $\Gamma_{I,J}$, i.e., $H^i_{I,J}(\mathscr M):=\mathbb R^i\Gamma_{I,J}(\mathscr M)$. 

\smallskip
When $J=0$, we set $H^i_{I}:=H^i_{I,0}(\mathscr M)$ and refer to $H^\bullet_I(\mathscr M)$ as the local cohomology objects with respect to $I$. 
\end{defn} 

We will now show that if $\mathscr M\in \mathscr S_R$ is $(I,J)$-torsion, then $H^{i}_{I,J}(\mathscr M)=0$ for $i>0$. This will be achieved by a sequence of steps.
First, we will show that the torsion theory $(\mathscr T(\mathscr S)_{I,J},\mathscr F(\mathscr S)_{I,J})$ satisfies the property that the torsion part of an injective in $\mathscr S_R$ is still an injective.  

\begin{lem}\label{L4.85fq}
Let $\mathscr E\in \mathscr S_R$ be an injective object. Then, for any ideal $I\subseteq R$, $\Gamma_I(\mathscr E)=\Gamma_{I,0}(\mathscr E)$ is also an injective object in $\mathscr S_R$.
\end{lem}

\begin{proof}
We will show that for any finitely generated $\mathscr M\in \mathscr S_R$ and any subobject $\mathscr N\subseteq \mathscr M$, a morphism $\phi:\mathscr N\longrightarrow \Gamma_I(\mathscr E)$ extends to a morphism $\psi:\mathscr M\longrightarrow \Gamma_I(\mathscr E)$. Since $\mathscr S_R$ is locally noetherian, it has a set of finitely generated generators. Accordingly, it will follow from Baer's criterion \cite[Prop V.2.9]{Sten} that $\Gamma_I(\mathscr E)$ is injective. 

\smallskip
Since $\phi(\mathscr N)\subseteq \Gamma_I(\mathscr E) \subseteq \mathscr E$ is finitely generated, we can choose $t\geq 1$ such that $I^t\phi(\mathscr N)=0$.
Since $\mathscr E$ is injective, the morphism $\mathscr N\overset{\phi}{\longrightarrow}\Gamma_I(\mathscr E)\hookrightarrow \mathscr E$ extends to a morphism $\psi': \mathscr M\longrightarrow \mathscr E$. In particular, $\phi(\mathscr N)\subseteq \psi'(\mathscr M)$. Since $\mathscr M$ is finitely generated, so is $\psi'(\mathscr M)$. Applying the version of Artin-Rees lemma   proved in \cite[Proposition D1.8]{AZ}, we can find $c\geq 1$ such that for all $n\geq c$ we have
\begin{equation}\label{4.5qa}
I^n\psi'(\mathscr M)\cap \phi(\mathscr N)=I^{n-c}(I^c\psi'(\mathscr M)\cap \phi(\mathscr N))\subseteq I^{n-c}\phi(\mathscr N)
\end{equation} Putting $n=t+c$, it follows from \eqref{4.5qa} that $I^{t+c}\psi'(\mathscr M)\cap \phi(\mathscr N)\subseteq I^t\phi(\mathscr N)=0$. We now note that
\begin{equation}\label{4.55t}
\phi(I^{t+c}\mathscr M\cap \mathscr N)\subseteq \phi(\mathscr N)\qquad \phi(I^{t+c}\mathscr M\cap \mathscr N)\subseteq \psi'(I^{t+c}\mathscr M\cap \mathscr N)\subseteq \psi'(I^{t+c}\mathscr M)\subseteq I^{t+c}\psi'(\mathscr M)
\end{equation} Accordingly, it follows from \eqref{4.55t} that $\phi(I^{t+c}\mathscr M\cap \mathscr N)\subseteq I^{t+c}\psi'(\mathscr M)\cap \phi(\mathscr N)=0$.  
Now since
\begin{equation}
I^{t+c} \mathscr M\cap \mathscr N=Ker(I^{t+c}\mathscr M\oplus \mathscr N\longrightarrow \mathscr M) \qquad  I^{t+c}\mathscr M+\mathscr N=Im(I^{t+c}\mathscr M\oplus\mathscr N
\longrightarrow \mathscr M) 
\end{equation} the morphism $0\oplus \phi:  I^{t+c} \mathscr M\oplus \mathscr N\longrightarrow \Gamma_I(\mathscr E)$ induces $\phi': I^{t+c} \mathscr M+ \mathscr N\longrightarrow \Gamma_I(\mathscr E)$  such that $\phi'| I^{t+c}\mathscr M=0$ and $\phi'|\mathscr N=\phi$. Again since $\mathscr E$ is injective, there exists $\psi: \mathscr M\longrightarrow \mathscr E$ extending $ I^{t+c} \mathscr M+ \mathscr N\overset{\phi'}{\longrightarrow} \Gamma_I(\mathscr E)\hookrightarrow \mathscr E$. We now see that 
\begin{equation}
I^{t+c}\psi(\mathscr M)  =\psi(I^{t+c}\mathscr M)=\phi'(I^{t+c}\mathscr M)=0
\end{equation} whence it follows that $\psi(\mathscr M)\subseteq \Gamma_I(\mathscr E)$. This proves the result.
\end{proof}

The next step is to make use of the directed sets $\tilde{\mathbb W}(I,J)$ defined in \cite[$\S$ 3]{TYY} as follows: the elements of $\tilde{\mathbb W}(I,J)$ are ideals $K\subseteq R$ satisfying $I^n
\subseteq K+J$ for $n\gg 1$. For $K$, $K'\in \tilde{\mathbb W}(I,J)$, we will say that $K\leq K'$ if $K\supseteq K'$. We observe that each $\tilde{\mathbb W}(I,J)$  filtered. 

\begin{lem}\label{L4.9sl}
Let $\mathscr M\in \mathscr S_R$. Then, $\Gamma_{I,J}(\mathscr M)=\underset{K\in \tilde{\mathbb W}(I,J)}{\varinjlim}\textrm{ }\Gamma_{K}(\mathscr M)$.
\end{lem}

\begin{proof}
We set $\mathscr N:=\underset{K\in \tilde{\mathbb W}(I,J)}{\varinjlim}\textrm{ }\Gamma_{K}(\mathscr M)$ and consider some $\mathscr N'\in fg(\mathscr N)$.  Since $\tilde{\mathbb W}(I,J)$ is filtered, we choose $K\in \tilde{\mathbb W}(I,J)$ such that $\mathscr N'\subseteq \Gamma_{K}(\mathscr M)$, i.e., $K^m\subseteq Ann(\mathscr N')$ for some $m\geq 1$. Since $K\in \tilde{\mathbb W}(I,J)$, we may choose $n\geq 1$ such that $I^n\subseteq K+J$. Then, $I^{mn}\subseteq K^m+J\subseteq Ann(\mathscr N')+J$ and hence $\mathscr N'\subseteq \Gamma_{I,J}(\mathscr M)$.
Conversely, we consider some $\mathscr N''\in fg(\Gamma_{I,J}(\mathscr M))$. Then, we have $I^l\subseteq Ann(\mathscr N'')+J$ for some $l\geq 1$. Then, $Ann(\mathscr N'')\in 
\tilde{\mathbb W}(I,J)$ and we have $\mathscr N''\subseteq \Gamma_{Ann(\mathscr N'')}(\mathscr M)$. Accordingly, we have $fg(\Gamma_{I,J}(\mathscr M))=fg(\mathscr N)$ and hence
$\Gamma_{I,J}(\mathscr M)=\mathscr N$ because $\mathscr S_R$ is locally noetherian. 
\end{proof}

\begin{thm}\label{P4.10gst}
Let $\mathscr E\in \mathscr S_R$ be an injective object. Then, for any ideals $I,J\subseteq R$, $\Gamma_{I,J}(\mathscr E)$ is also an injective object in $\mathscr S_R$.
\end{thm} 

\begin{proof}
Since $\mathscr E\in \mathscr S_R$ is injective, it follows by Lemma \ref{L4.85fq} that $\Gamma_K(\mathscr E)$ is injective for each $K\in \tilde{\mathbb W}(I,J)$. We now consider some finitely generated $\mathscr M\in \mathscr S_R$, a subobject $\mathscr N\subseteq \mathscr M$ and a morphism $\phi:\mathscr N\longrightarrow \Gamma_{I,J}(\mathscr E)$. By Lemma \ref{L4.9sl}, we know that $\Gamma_{I,J}(\mathscr E)=\underset{K\in \tilde{\mathbb W}(I,J)}{\varinjlim}\textrm{ }\Gamma_{K}(\mathscr E)$. Since $\tilde{\mathbb W}(I,J)$ is filtered and $\mathscr N\subseteq \mathscr M$ is finitely generated, we can find $K_0\in \tilde{\mathbb W}(I,J)$ such that $\phi$ factors through $\Gamma_{K_0}(\mathscr E)$. Since $\Gamma_{K_0}(\mathscr E)$ is injective, we now have $\psi':\mathscr M\longrightarrow \Gamma_{K_0}(\mathscr E)$ extending $\mathscr N\longrightarrow \Gamma_{K_0}(\mathscr E)$. Composing with the canonical map $\Gamma_{K_0}(\mathscr E)\longrightarrow \underset{K\in \tilde{\mathbb W}(I,J)}{\varinjlim}\textrm{ }\Gamma_{K}(\mathscr M)=\Gamma_{I,J}(\mathscr E)$, we obtain an extension of $\phi:\mathscr N\longrightarrow \Gamma_{I,J}(\mathscr E)$ to a morphism $\psi:\mathscr M\longrightarrow \Gamma_{I,J}(\mathscr E)$. Since $\mathscr S_R$ is a locally noetherian category, it now follows from Baer's criterion (see  \cite[Prop V.2.9]{Sten}) that $\Gamma_{I,J}(\mathscr E)$ is injective.
\end{proof}

\begin{lem}\label{L4.7x4}
Let $\mathscr M\in \mathscr S_R$ be an $(I,J)$-torsion object. Then, its injective hull $\mathscr E(\mathscr M)$ is also $(I,J)$-torsion.
\end{lem}

\begin{proof} Since $\mathscr E(\mathscr M)$ is injective, it follows by  Proposition \ref{P4.10gst}  that $\Gamma_{I,J}(\mathscr E(\mathscr M))$ is also injective in $\mathscr S_R$. Accordingly, we have a direct sum decomposition $\mathscr E(\mathscr M)=\Gamma_{I,J}(\mathscr E(\mathscr M))\oplus \mathscr N$. We already know that $\mathscr M\in \mathscr S_R$ is an $(I,J)$-torsion object and it is clear from the definition in \eqref{4.1dn} that $\mathscr M$ lies inside the torsion part of $\mathscr E(\mathscr M)$. Then, $\mathscr M\cap \mathscr N=0$. But $\mathscr M$ is essential in $\mathscr E(\mathscr M)$, whence it follows that $\mathscr N=0$, i.e., $\mathscr E(\mathscr M)=\Gamma_{I,J}(\mathscr E(\mathscr M))$. 

\end{proof}

\begin{thm}\label{P4.9x4}
Let $\mathscr M\in \mathscr S_R$ be an $(I,J)$-torsion object. Then, $H^i_{I,J}(\mathscr M)=0$ for $i>0$.
\end{thm}

\begin{proof}
By Lemma \ref{L4.7x4}, we know that the injective hull of an $(I,J)$-torsion object is also $(I,J)$-torsion. By Proposition \ref{torth}, we know that the torsion class 
$\mathscr T(\mathscr S)_{I,J}$ is closed under both quotients and subobjects. Accordingly,   we can obtain a resolution $\mathscr M\longrightarrow \mathscr E^\bullet$ consisting of injective objects that are $(I,J)$-torsion. Therefore, applying the functor $\Gamma_{I,J}$ leaves this sequence unchanged. By definition, the cohomology of the resolution $\mathscr E^\bullet$ vanishes in positive degrees, and this proves the result.
\end{proof}

\begin{cor}\label{C410x4}
Let $\mathscr M\in \mathscr S_R$. Then, 

\smallskip
(a) $H^i_{I,J}(\mathscr M)\cong H^i_I(\mathscr M/\Gamma_{I,J}(\mathscr M)) $ for $i>0$.

\smallskip
(b)  $H^i_{I,J}(\mathscr M)$ is $(I,J)$-torsion for any $i\geq 0$.
\end{cor}

\begin{proof}
By Proposition \ref{P4.9x4}, we know that $H^i_I(\Gamma_{I,J}(\mathscr M))=0$ for $i>0$. The result of (a) now clear from the long exact sequence of cohomologies obtained by applying 
$\Gamma_{I,J}(\_\_)$ to the short exact sequence $0\longrightarrow \Gamma_{I,J}(\mathscr M)\longrightarrow \mathscr M\longrightarrow \mathscr M/\Gamma_{I,J}(\mathscr M))
\longrightarrow 0$. To prove (b), we note that $H^\bullet_{I,J}(\mathscr M)$ are by definition the homology objects of $\Gamma_{I,J}(\mathscr E^\bullet)$, where $\mathscr M\longrightarrow \mathscr E^\bullet$ is an injective resolution. Since each $\Gamma_{I,J}(\mathscr E^\bullet)$ lies in the hereditary torsion class $\mathscr T(\mathscr S)_{I,J}$, so does its subquotient
$H^\bullet_{I,J}(\mathscr M)$. 
\end{proof}

\section{Associated primes of local cohomology objects}

In this section, we study finiteness conditions on local cohomology objects and their associated primes. We give a condition for the set of associated primes of the local cohomology object $H^\bullet_{I,J}(\_\_)$ with respect to an ideal pair $(I,J)$ in an abstract module category  to be finite.
For $\mathscr M\in \mathscr S_R$, we say that $a\in R$ is a non-zerodivisor on $\mathscr M$ if $a_{\mathscr M}:\mathscr M\longrightarrow \mathscr M$ is a monomorphism.

\begin{lem}
\label{L5.1} Let $a\in R$ be a zero divisor on $\mathscr M\in \mathscr S_R$. Then, $a$ lies in some associated prime of $\mathscr M$.
\end{lem}

\begin{proof}
By definition, we know that $Ker(a_{\mathscr M})\ne 0$. Using Proposition \ref{P3.3},  we pick a prime ideal $\mathfrak p \in Ass(Ker(a_{\mathscr M}))\subseteq Ass(\mathscr M)$. Then, $\mathfrak p=Ann(\mathscr L)$ for some $R$-elementary object $\mathscr L\subseteq Ker(a_{\mathscr M})\subseteq \mathscr M$. Then, $a\in Ann(Ker(a_{\mathscr M}))\subseteq Ann(\mathscr 
L)=\mathfrak p$. 
\end{proof}

\begin{lem}
\label{L5.2} Let $I,J\subseteq R$ be ideals and let $\mathscr M\in \mathscr S_R$ be a finitely generated object such that $\Gamma_{I,J}(\mathscr M)=0$. Then, $I$ contains a non-zero divisor on $\mathscr M$.
\end{lem}

\begin{proof}Suppose that all elements of $I\subseteq R$ are zero divisors on $\mathscr M$. By Lemma \ref{L5.1}, $I\subseteq \underset{\mathfrak p\in Ass(\mathscr M)}{\bigcup}
\textrm{ }\mathfrak p$.  Since $\mathscr M$ is finitely generated, $Ass(\mathscr M)$ is finite by Corollary \ref{C3.6}. By prime avoidance, there is some $\mathfrak p_0\in Ass(\mathscr M)$ such that
$I\subseteq \mathfrak p_0$. Hence, $\mathfrak p_0\in \mathbb W(I,J)$. 

\smallskip Now since $\mathfrak p_0$ is an associated prime of $\mathscr M$, there is an $R$-elementary object $0\ne \mathscr L\subseteq \mathscr M$ such that $\mathfrak p_0=Ann(\mathscr L)$.  Then, $Supp(\mathscr L)=\mathbb V(\mathfrak p_0)\subseteq \mathbb W(I,J)$ and hence $\mathscr L\subseteq \Gamma_{I,J}(\mathscr M)=0$, which is a contradiction.
\end{proof}

For any $\mathscr M$ in $\mathscr S_R$ and subobjects 
$\mathscr M'$, $\mathscr M''\subseteq \mathscr M$, we set
\begin{equation}
\label{5.1ew}
(\mathscr M'':\mathscr M'):=\{\mbox{$a\in R$ $\vert$ $Im(a_{\mathscr M'})\subseteq \mathscr M''$}\}
\end{equation} It is clear that any such $(\mathscr M'':\mathscr M')$ is an ideal in $R$.

\begin{lem}\label{L5.29}
Let $\mathscr N\subseteq \mathscr M$ in $\mathscr S_R$ and let $\mathscr K\subseteq \mathscr M/\mathscr N$ be a finitely generated subobject. Then, there exists a finitely generated
subobject $\mathscr K'\subseteq\mathscr M$ such that 

\smallskip
(1) The quotient map 
$\pi:\mathscr M\twoheadrightarrow \mathscr M/\mathscr N$ restricts to an epimorphism $\pi':\mathscr K'\twoheadrightarrow \mathscr K$.

\smallskip
(2) $Ann(\mathscr K)=(\mathscr N:\mathscr K')$.
\end{lem}

\begin{proof}
We put $\mathscr K'':=\pi^{-1}(\mathscr K)\subseteq \mathscr M$. Since finitely generated subobjects of $\mathscr K''$ form a filtered system and $\pi(\mathscr K'')=\mathscr K$, we can find
$\mathscr K'\in fg(\mathscr K'')$ such that $\pi(\mathscr K')=\mathscr K$. This proves (1). To prove (2), for any $a\in R$, we  consider the commutative diagram
\begin{equation}\label{52cd}
\begin{CD}
\mathscr K'@>a_{\mathscr K'}>> \mathscr K' @>\iota'>> \mathscr M\\
@V\pi' VV @V\pi' VV @VV\pi V\\
 \mathscr K@>a_{\mathscr K}>> \mathscr K@>\iota>> \mathscr M/\mathscr N\\
 \end{CD}
\end{equation} If $a\in Ann(\mathscr K)$, i.e., $a_{\mathscr K}=0$, we get $\pi\circ\iota'\circ a_{\mathscr K'}=\iota\circ a_{\mathscr K}\circ \pi'=0$, which gives $Im(a_{\mathscr K'})\subseteq Ker(\pi)=\mathscr N$, i.e., $a\in (\mathscr N:\mathscr K')$. Conversely, if $a\in (\mathscr N:\mathscr K')$, we get $\iota\circ a_{\mathscr K}\circ \pi'=0$. Since $\pi'$ is an epimorphism and $\iota$ is a monomorphism, we get $a_{\mathscr K}=0$ or $a\in Ann(\mathscr K)$.  
\end{proof}

\begin{thm}\label{T5.3}
Let $\mathscr M\in \mathscr S_R$ be finitely generated and let $I,J\subseteq R$ be  ideals. Suppose that $i\geq 0$ is such that $H^j_{I,J}(\mathscr M)$ is finitely generated for all $j<i$. Then, for any finitely generated $\mathscr N\subseteq H^i_{I,J}(\mathscr M)$ such that $\mathscr N\supseteq JH^i_{I,J}(\mathscr M)$, the collection $Ass(H^i_{I,J}(\mathscr M)/\mathscr N)$ is finite.
\end{thm}

\begin{proof} 
For $i=0$, we know that $H^0_{I,J}(\mathscr M)=\Gamma_{I,J}(\mathscr M)\subseteq \mathscr M$ is finitely generated and the result is clear. We will proceed by induction on $i$. By Corollary \ref{C410x4}, we know that $H^j_{I,J}(\mathscr M)\cong H^j_{I,J}(\mathscr M/\Gamma_{I,J}(\mathscr M)) $ for $j>0$. Also, we know that $\Gamma_{I,J}(\mathscr M/\Gamma_{I,J}(\mathscr M))=0$. Accordingly, we may suppose that $\mathscr M\in \mathscr S_R$ is finitely generated with $\Gamma_{I,J}(\mathscr M)=0$  and it follows by Lemma \ref{L5.2} that we can find $a\in I$ which is a non-zero divisor on $\mathscr M$. 

\smallskip
By Corollary \ref{C410x4}, we know that $H^i_{I,J}(\mathscr M)$ is $(I,J)$-torsion, i.e., $\Gamma_{I,J}(H^i_{I,J}(\mathscr M))=H^i_{I,J}(\mathscr M)$. Since $\mathscr N\subseteq H^i_{I,J}(\mathscr M)$ is finitely generated, it follows by Lemma \ref{L4.1qu} that $I^t\subseteq Ann(\mathscr N)+J$ for some $t>0$. Accordingly, we have $a^t\mathscr N\subseteq J\mathscr N$. We now consider the short exact sequence
\begin{equation}\label{ses5.2}
0\longrightarrow \mathscr M \overset{a^t}{\longrightarrow}\mathscr M\longrightarrow\mathscr M/a^t\mathscr M\longrightarrow 0
\end{equation} Applying $\Gamma_{I,J}(\_\_)$ to \eqref{ses5.2}, the long exact sequence of derived functors gives us 
\begin{equation}\label{les5.3}
H^{l}_{I,J}(\mathscr M)\overset{a^t}{\longrightarrow} H^{l}_{I,J}(\mathscr M)\longrightarrow H^l_{I,J}(\mathscr M/a^t\mathscr M)\longrightarrow H^{l+1}_{I,J}(\mathscr M)\overset{a^t}{\longrightarrow}H^{l+1}_{I,J}(\mathscr M)
\end{equation} for every $l\geq 0$. Since $\mathscr S_R$ is locally noetherian, the collection of finitely generated objects in $\mathscr S_R$ is closed under extensions, subobjects and quotients (see \cite[V.4.2]{Sten}). Since $H^j_{I,J}(\mathscr M)$ is finitely generated for all $j<i$, it follows from \eqref{les5.3} that $H^l_{I,J}(\mathscr M/a^t\mathscr M)$ is finitely generated for $l<i-1$. 

\smallskip
Since $a^t\mathscr N\subseteq J\mathscr N$, we also have the following diagram where the top row is exact and the vertical morphisms are epimorphisms
\begin{equation}\label{cdes5.4}
\begin{CD}
H^{i-1}_{I,J}(\mathscr M) @>\iota>> H^{i-1}_{I,J}(\mathscr M/a^t\mathscr M) @>\delta>> H^i_{I,J}(\mathscr M) @>\mu:=a^t>> H^i_{I,J}(\mathscr M)\\
@. @. @V{\pi}VV @VV\pi'V \\
 @.  @.  H^i_{I,J}(\mathscr M)/\mathscr N@>\bar{\mu}:=\bar{a}^t>>  H^i_{I,J}(\mathscr M)/  J\mathscr N\\
\end{CD}
\end{equation}   We now observe that $Ker(\bar\mu)
=\mu^{-1}(J\mathscr N)/\mathscr N$.  This gives us the short exact sequence
\begin{equation}\label{5.9sesc}\small 
 0\longrightarrow H^{i-1}_{I,J}(\mathscr M/a^t\mathscr M)/\delta^{-1}(\mathscr N)=(Im(\delta)+\mathscr N)/\mathscr N\longrightarrow Ker(\bar\mu)
=\mu^{-1}(J\mathscr N)/\mathscr N\longrightarrow \mu^{-1}(J\mathscr N)/(Ker(\mu)+\mathscr N)\longrightarrow 0
\end{equation} where we have used $Ker(\mu)=Im(\delta)$. Since $J\mathscr N\subseteq \mathscr N$ is finitely generated, so is its subobject $\mu^{-1}(J\mathscr N)/Ker(\mu)$. Then, the quotient $\mu^{-1}(J\mathscr N)/(Ker(\mu)+\mathscr N)$ of $\mu^{-1}(J\mathscr N)/Ker(\mu)$ is also finitely generated. 

\smallskip On the other hand, we have the short exact sequence
\begin{equation}\label{ses5.5}
0\longrightarrow Im(\iota)\cap \delta^{-1}(\mathscr N)= Ker(\delta)\cap \delta^{-1}(\mathscr N)\longrightarrow \delta^{-1}(\mathscr N)\overset{\delta}{\longrightarrow} \delta(\delta^{-1}(\mathscr N))\longrightarrow 0
\end{equation} By assumption, $H^{i-1}_{I,J}(\mathscr M)$ is finitely generated and hence so is its subquotient $Im(\iota)\cap \delta^{-1}(\mathscr N)$. Also since $\mathscr N$ is finitely generated, so is $\delta(\delta^{-1}(\mathscr N))\subseteq \mathscr N$. Again since $\mathscr S_R$ is locally noetherian, it follows from \eqref{ses5.5} that $\delta^{-1}(\mathscr N)$
is finitely generated.  Additionally, we note that
\begin{equation}\label{5.12dj}
\delta(JH^{i-1}_{I,J}(\mathscr M/a^t\mathscr M))\subseteq J\delta(H^{i-1}_{I,J}(\mathscr M/a^t\mathscr M))\subseteq JH^i_{I,J}(\mathscr M) \subseteq \mathscr N\qquad \Rightarrow \qquad JH^{i-1}_{I,J}(\mathscr M/a^t\mathscr M)\subseteq \delta^{-1}(\mathscr N)
\end{equation}

We now put $\mathscr K:= H^{i-1}_{I,J}(\mathscr M/a^t\mathscr M) /\delta^{-1}(\mathscr N)$ and $\mathscr P:=Ker(\bar\mu)$. By the induction assumption, we see that $Ass(\mathscr K)$ is a finite set.  Since $\mu^{-1}(J\mathscr N)/(Ker(\mu)+\mathscr N)$ is finitely generated,  from the short exact sequence \eqref{5.9sesc}, it now follows from Proposition \ref{P3.4} that $Ass(\mathscr P)$ is finite.

\smallskip Since $\mathscr N$ is finitely generated, so is its image $\pi'(\mathscr N)$. To prove the result, it now suffices to show  that
\begin{equation}\label{5.6t}
Ass(H^i_{I,J}(\mathscr M)/\mathscr N)\subseteq Ass(\mathscr P)\cup Ass(\pi'(\mathscr N))
\end{equation} Accordingly, we choose some prime ideal $\mathfrak p\in Ass(H^i_{I,J}(\mathscr M)/\mathscr N)\backslash Ass(\mathscr P)$. Then there is an $R$-elementary object $\mathscr L\subseteq H^i_{I,J}(\mathscr M)/\mathscr N$ such that $\mathfrak p=Ann(\mathscr L)$. Applying Lemma \ref{L5.29}, we can choose a finitely generated subobject $\mathscr L'\subseteq H^i_{I,J}(\mathscr M)$ such that $\pi:H^i_{I,J}(\mathscr M)\longrightarrow H^i_{I,J}(\mathscr M)/\mathscr N$ restricts to an epimorphism $\mathscr L'\twoheadrightarrow\mathscr L$ and $(\mathscr N:\mathscr L')=Ann(\mathscr L)=\mathfrak p$. We now consider the short exact sequence
\begin{equation}\label{5.7f}
\begin{CD}
0@>>> \mathscr P\cap \mathscr L@>>> \mathscr L@>\bar\mu|\mathscr L=\bar{a}^t>>\bar{a}^t\mathscr L@>>> 0
\end{CD}
\end{equation} By Proposition \ref{P3.4}, we get $Ass(\mathscr L)\subseteq Ass(\mathscr P\cap \mathscr L)\cup  Ass(\bar{a}^t\mathscr L)
\subseteq Ass(\mathscr P)\cup  Ass(\bar{a}^t\mathscr L)$. Since $\mathfrak p\in Ass(\mathscr L)$ does not lie in $Ass(\mathscr P)$, we obtain $\mathfrak p\in  Ass(\bar{a}^t\mathscr L)=Ass(\bar{a}^t\pi(\mathscr L'))=Ass(\pi' a^t\mathscr L')$.  

\smallskip
We know that $a^t\mathscr L'\subseteq \mathscr L'$ is finitely generated. Again since $H^i_{I,J}(\mathscr M)$ is $(I,J)$-torsion, it follows that $a^{t+s}\mathscr L'\subseteq J\mathscr L'\subseteq JH^i_{I,J}(\mathscr M)\subseteq \mathscr N$ for some $s\geq 1$.   But since $(\mathscr N:\mathscr L')=\mathfrak p$, we get $a^{t+s}\in \mathfrak p$, i.e., $a\in \mathfrak p$. Again since $\mathfrak p=(\mathscr N:\mathscr L')$, we obtain $a^t\mathscr L'\subseteq \mathscr N$ and hence
$\pi' a^t\mathscr L'\subseteq \pi'(\mathscr N)$. Then, $\mathfrak p\in Ass(\pi'a^t\mathscr L')\subseteq Ass(\pi'(\mathscr N))$. 
\end{proof}

\begin{Thm}\label{TBlas}
Let $\mathscr M\in \mathscr S_R$ be finitely generated and let $I,J\subseteq R$ be  ideals. Suppose that $i\geq 0$ is such that $H^j_{I,J}(\mathscr M)$ is finitely generated for all $j<i$. Suppose that $JH^i_{I,J}(\mathscr M)$ is finitely generated. Then, for any finitely generated $\mathscr N\subseteq H^i_{I,J}(\mathscr M)$, the collection $Ass(H^i_{I,J}(\mathscr M)/\mathscr N)$ is finite.
\end{Thm}

\begin{proof}
Since $\mathscr N$ and $JH^i_{I,J}(\mathscr M)$ are both finitely generated, so is $\mathscr N+JH^i_{I,J}(\mathscr M)$. By Proposition \ref{T5.3}, it now follows that $Ass(H^i_{I,J}(\mathscr M)/(\mathscr N+JH^i_{I,J}(\mathscr M))$ is finite. We now consider the short exact sequence
\begin{equation}\label{515sesd}
0\longrightarrow (\mathscr N+JH^i_{I,J}(\mathscr M))/\mathscr N\longrightarrow  H^i_{I,J}(\mathscr M)/\mathscr N\longrightarrow H^i_{I,J}(\mathscr M)/(\mathscr N+JH^i_{I,J}(\mathscr M))\longrightarrow 0
\end{equation} Since $ (\mathscr N+JH^i_{I,J}(\mathscr M))/\mathscr N$ is finitely generated and $Ass(H^i_{I,J}(\mathscr M)/(\mathscr N+JH^i_{I,J}(\mathscr M))$ is finite, the result follows by applying Proposition \ref{P3.4} to the short exact sequence \eqref{515sesd}.
\end{proof}

\begin{cor}\label{C5.4f}
Let $\mathscr M\in \mathscr S_R$ be finitely generated and let $I,J\subseteq R$ be  ideals. Suppose that $i\geq 0$ is such that $H^j_{I,J}(\mathscr M)$ is finitely generated for all $j<i$. Suppose that $JH^i_{I,J}(\mathscr M)$ is finitely generated. Then, the collection $Ass(H^i_{I,J}(\mathscr M))$ is finite.
\end{cor}

\begin{proof} This follows directly from Theorem \ref{TBlas} by setting $\mathscr N=0$.
\end{proof}

We conclude this section with the following fact, which extends a result of Brodmann, Rotthaus and Sharp \cite{BRS}. 

\begin{thm}
Let $\mathscr M\in \mathscr S_R$ be such that the collection $\chi$ of maximal elements in $Ass(\mathscr M)$ is finite. Let $I\subseteq R$ be an ideal. Suppose that for each $\mathfrak p\in \chi$, there is 
$n_{\mathfrak p}>0$ such that $I^{n_{\mathfrak p}}\mathscr M_{\mathfrak p}=0$. If $n:=max\{\mbox{$n_{\mathfrak p}$ $\vert$ $\mathfrak p\in \chi$}\}$, then
$I^n\mathscr M=0$. 
\end{thm}

\begin{proof}
We consider some $\mathscr N\in fg(\mathscr M)$. It is clear that $(I^n\mathscr N)_{\mathfrak p}\subseteq I^n\mathscr M_{\mathfrak p}=0$. Since $\mathscr S_R$ is locally noetherian, 
we know that $I^n\mathscr N\subseteq \mathscr N$ is finitely generated. From \eqref{kerun}, we see that for each $\mathfrak p\in \chi$, we have
\begin{equation}\label{kerunsk}
(I^n\mathscr N)=Ker\left((I^n\mathscr N)\longrightarrow (I^n\mathscr N)_{\mathfrak p}=0\right)=\underset{t\in R\backslash \mathfrak p}{\bigcup}\textrm{ }Ker((I^n\mathscr N)\overset{t}{\longrightarrow} (I^n\mathscr N))
\end{equation}  Since the union on the right hand side of \eqref{kerunsk} is filtered, we can find $s_{\mathfrak p}\in R\backslash \mathfrak p$ such that $s_{\mathfrak p}
\in Ann(I^n\mathscr N)$. Let $K\subseteq R$ be the ideal generated by the collection $\{s_{\mathfrak p}\}_{\mathfrak p\in \chi}$. It is clear that 
$K\subseteq Ann(I^n\mathscr N)$.   Since $K$ is not contained in any of the prime ideals in 
$\chi$, and $\chi$ is finite, it follows by prime avoidance that we can choose an element  $r\in K\backslash \underset{\mathfrak p\in \chi}{\bigcup}\textrm{ }\mathfrak p$. From Lemma 
\ref{L5.1}, it is clear that  $r$ must be a non-zero divisor on 
$\mathscr M$. But multiplication by $r$ is zero on the subobject $I^n\mathscr N\subseteq \mathscr M$. Hence, $I^n\mathscr N=0$. It follows that 
$I^n\mathscr M=\underset{\mathscr N\in fg(\mathscr M)}{\bigcup}I^n\mathscr N=0$. 
\end{proof}

\section{Spectral sequences for local cohomology type functors on $\mathscr S_R$}

In this section, we will show that we can construct spectral sequences for ``local cohomology type'' functors on $\mathscr S_R$ by creating an axiomatic setup similar to \`{A}lvarez Montaner, Boix and Zarzuela \cite{Alz}. Our main tools will be the properties of injective hulls of $R$-elementary objects proved in Section 2. We will then exhibit three different situations where this axiomatic setup can be used to obtain spectral sequences. 

\smallskip
Let $P$ be a finite poset and let $Fun(P,\mathscr S_R)$ denote the category of systems of objects in $\mathscr S_R$ indexed over $P$. We suppose from now on  that $\mathscr S$ has enough projectives. Then, it follows from Artin and Zhang \cite[Lemma D3.2]{AZ} that $\mathscr S_R$ has enough projectives. Then we know that $Fun(P,\mathscr S_R)$ has enough projectives and $Fun(P^{op},\mathscr S_R)$ has enough injectives (see, for instance, \cite[$\S$ 2.3]{Web}). We also  note that $Fun(P,\mathscr S_R)$ and $Fun(P^{op},
\mathscr S_R)$ are Grothendieck categories.

\smallskip
For any functor $\mathfrak G\in Fun(P^{op},\mathscr S_R)$, we consider the  complex $(C^\bullet(\mathfrak G),\partial^\bullet)$ given by 
\begin{equation}\label{Rooscoh}
C^k(\mathfrak G):=\underset{p_0< p_1<...< p_k}{\prod}\mathfrak G(p_0)\qquad\qquad \partial^k=\underset{j=0}{\overset{k+1}{\sum}}
(-1)^j\partial_j^k:C^k(\mathfrak G)\longrightarrow C^{k+1}(\mathfrak G)
\end{equation} where $\partial^k_j$ is induced by deleting the $j$-th term in the sequence $\{p_0<p_1< ...< p_k< p_{k+1}\}$. Because $P$ is a finite poset, it follows from the proof of \cite[Lemma A.3.2]{Neem} that this complex computes the derived functor of the inverse limit over $P$, i.e., for any $\mathfrak G\in Fun(P^{op},\mathscr S_R)$, we have $H^i(C^\bullet(\mathfrak G))=\mathbb R^i\lim_{p\in P}\mathfrak G(p)$.

\smallskip
Similarly, for any functor $\mathfrak F\in Fun(P,\mathscr S_R)$, we consider the complex the  complex $(C_\bullet(\mathfrak F),\partial_\bullet)$ given by
\begin{equation}\label{Roosho}
C_k(\mathfrak F):=\underset{p_0< p_1< ...< p_k}{\bigoplus}\mathfrak F(p_0)\qquad\qquad \partial_k=\underset{j=0}{\overset{k+1}{\sum}}
(-1)^j\partial^j_k:C_{k+1}(\mathfrak F)\longrightarrow C_{k}(\mathfrak F)
\end{equation} where $\partial_k^j$ is induced by deleting the $j$-th term in the sequence $\{p_0< p_1< ...< p_k< p_{k+1}\}$. Then, it follows from \cite[B.1.2]{Neem} that  this complex computes the derived functor of the direct limit over $P$, i.e., for any $\mathfrak F\in Fun(P,\mathscr S_R)$, we have $H_i(C_\bullet(\mathfrak F))=\mathbb L_i\text{colim}_{p\in P}\mathfrak F(p)$.

\smallskip
From now onwards, we fix an ideal $I\subseteq R$. We suppose that $I$ may be expressed as $I=I_{\alpha_1}\cap ... \cap I_{\alpha_k}$, where $I_{\alpha_1}$,...,$I_{\alpha_k}$ are ideals of $R$. As in \cite[$\S$ 2]{Alz}, we let $P$ denote the finite poset whose elements are all the possible different sums of the ideals $I_1$, ...,$I_n$, ordered by reverse inclusion. By $\hat{P}$, we will mean the poset $\hat{P}=
P\cup \{1_{\hat{P}}\}$ obtained by adding a  final element $1_{\hat{P}}$ to $P$ (even if $P$ already has a final element).  
The Alexandrov topology on ${P}$ (see \cite[$\S$ 2.2.1]{Alz}) is that whose open sets $U\subseteq {P}$ satisfy the property  that $p\in U$ and $p\leq q$ implies $q\in U$. If ${P}$ has the Alexandrov topology, we know (see \cite[Lemma 2.9]{Alz}) that the basic open set $[p,1_{\hat{P}})=\{\mbox{$q\in P$ $\vert$ $q\geq p$}\}$ is contractible for each $p\in P$.

\smallskip
For any $p\in P$, we will denote by $I_p$ the sum of ideals corresponding to $p\in P$. 
We now fix an ideal $J$ and an additive functor $\Psi_{[\ast]}:\mathscr S_R\longrightarrow Fun(\hat{P},\mathscr S_R)$ that satisfies the following conditions analogous to \cite[Setup 4.3]{Alz}. 

\smallskip
(1) For each $p\in \hat{P}$, the functor $\Psi_p:=\Psi_{[\ast]}(\_\_)(p):\mathscr S_R\longrightarrow \mathscr S_R$ is left exact and preserves direct sums. 

\smallskip
(2) Let $\mathscr L\in \mathscr S_R$ be an $R$-elementary object. Let $\mathscr E(\mathscr L)\in \mathscr S_R$ be the injective hull of $\mathscr L$ and let $\mathfrak{p}:=Ann(\mathscr N)$. Then, for any maximal ideal $\mathfrak m$ in $R$, there are objects $\mathscr X(\mathscr L,\mathfrak{m})$ and 
$\mathscr Y(\mathscr L,\mathfrak{m})\in \mathscr S_R$ such that for $p\in \hat{P}$, we have
\begin{equation}\label{eexxyy}
\Psi_p(\mathscr E(\mathscr L))_{\mathfrak m}=\left\{\begin{array}{ll}
\mathscr X(\mathscr L,\mathfrak{m}) & \mbox{if $\mathfrak p\in \mathbb W(I_p,J)$ and $\mathfrak p\subseteq \mathfrak m$} \\
\mathscr Y(\mathscr L,\mathfrak{m}) & \mbox{if $\mathfrak p\notin \mathbb W(I_p,J)$ and $\mathfrak p \subseteq \mathfrak m$} \\
0 & \mbox{otherwise} \\
\end{array}\right.
\end{equation} 

(3) For any $p\leq q$, the natural transformation $\Psi_p\longrightarrow \Psi_q$ satisfies, for any $R$-elementary object $\mathscr L\in \mathscr S_R$ and 
maximal ideal $\mathfrak m$:
\begin{equation}\label{ey5.4c}
\Psi_p(\mathscr E(\mathscr L))_{\mathfrak m}\longrightarrow \Psi_q(\mathscr E(\mathscr L))_{\mathfrak m}=\left\{
\begin{array}{ll}
id_{\mathscr X(\mathscr L,\mathfrak{m})}&\mbox{if $\Psi_p(\mathscr E(\mathscr L))_{\mathfrak m}=\mathscr X(\mathscr L,\mathfrak{m})=\Psi_q(\mathscr E(\mathscr L))_{\mathfrak m}$} \\
id_{\mathscr Y(\mathscr L,\mathfrak{m})}& \mbox{if $\Psi_p(\mathscr E(\mathscr L))_{\mathfrak m}=\mathscr Y(\mathscr L,\mathfrak{m})=\Psi_q(\mathscr E(\mathscr L))_{\mathfrak m}$} \\
0 & \mbox{otherwise} \\
\end{array}\right.
\end{equation}  We also set
$\Psi:=\Psi_{1_{\hat{P}}}:\mathscr S_R\longrightarrow \mathscr S_R$. Moreover, for $i\geq 0$ and any $\mathscr M\in \mathscr S_R$, we let $\mathbb R^i\Psi_{[\ast]}(\mathscr M)$ denote the direct system of derived functors $\{\mathbb R^i\Psi_p(\mathscr M)\}_{p\in P}$.

\begin{lem}\label{L5.1lp} Let $\mathscr M\in \mathscr S_R$ be such that the localization $\mathscr M_{\mathfrak m}=0$ for every  maximal ideal $\mathfrak m
\subseteq R$. Then, $\mathscr M=0$.
\end{lem}

\begin{proof}
We consider any $\mathscr M'\in fg(\mathscr M)$. Since localization is exact, we have inclusions $\mathscr M'_{\mathfrak m}\hookrightarrow \mathscr M_{\mathfrak m}=0$ and hence 
$\mathscr M'_{\mathfrak m}=0$ for each maximal ideal $\mathfrak m$. Since $\mathscr M'$ is finitely generated, it follows by Proposition \ref{L3.8uh} that $Supp(\mathscr M')=\mathbb V(Ann(\mathscr M'))$. Hence, we have $\mathbb V(Ann(\mathscr M'))=\phi$, i.e., $\mathscr M'=0$. Since $\mathscr S_R$ is locally noetherian, $\mathscr M$ is the sum of all its finitely generated subobjects. Hence, $\mathscr M=0$. 
\end{proof}

\begin{lem}\label{L5.2dm} Let $\mathscr E\in \mathscr S_R$ be an injective object.   Then, 
the complex $C_\bullet(\Psi_{[\ast]}(\mathscr E)) \longrightarrow  \Psi(\mathscr E ) \longrightarrow 0$ is exact.

\end{lem}

\begin{proof} Since localizations are exact, it follows from Lemma \ref{L5.1lp} that it suffices to show that $C_\bullet(\Psi_{[\ast]}(\mathscr E))_{\mathfrak m} \longrightarrow  \Psi(\mathscr E )_{\mathfrak m} \longrightarrow 0$ is exact for each maximal ideal $\mathfrak m$. Since $\mathscr E\in \mathscr S_R$ is injective, we know from Theorem \ref{T3.8bi} that $\mathscr E$ can be expressed as a direct sum of injective hulls of $R$-elementary objects. Since $\Psi_{[\ast]}$ and $C_\bullet$ preserve direct sums, it now suffices to check that $C_\bullet(\Psi_{[\ast]}(\mathscr E(\mathscr L)))_{\mathfrak m} \longrightarrow  \Psi(\mathscr E (\mathscr L))_{\mathfrak m} \longrightarrow 0$ is exact for each maximal ideal $\mathfrak m$, where $\mathscr L\in \mathscr S_R$ is an $R$-elementary object and $\mathscr E(\mathscr L)$ its injective hull. We now set $\mathfrak{p}:=Ann(\mathscr L)$. The rest of the proof now follows exactly as in \cite[Lemma 4.4]{Alz}. 
\end{proof}

\begin{Thm}\label{T5.3cr} Let $\mathscr S$ be a strongly locally noetherian Grothendieck category and let $R$ be a commutative noetherian ring. Then, for any $\mathscr M\in \mathscr S_R$, we have a spectral sequence
\begin{equation}
E_2^{-i,j}=\mathbb L_i\text{colim}_{p\in P}\mathbb R^j\Psi_{[\ast]}(\mathscr M)\Longrightarrow \mathbb R^{j-i}\Psi(\mathscr M)
\end{equation}
\end{Thm}

\begin{proof}
For $\mathscr M\in \mathscr S_R$, we consider an injective resolution $0\longrightarrow \mathscr M\longrightarrow \mathscr E^\bullet$. We now consider the following bicomplex $D^{\bullet,\bullet}$ that is concentrated in the ``second quadrant.''
\begin{equation}\label{DDCC}
\xymatrix{
&& &&  && \\
\dots\ar[rr] && C_1(\Psi_{[\ast]}(\mathscr E^1)) \ar[rr] \ar[u]&& C_0(\Psi_{[\ast]}(\mathscr E^1)) \ar[rr]\ar[u] && 0\\
\dots\ar[rr] && C_1(\Psi_{[\ast]}(\mathscr E^0)) \ar[rr] \ar[u] && C_0(\Psi_{[\ast]}(\mathscr E^0))\ar[u] \ar[rr] && 0}
\end{equation} Here, we write $D^{-l,k}:=C_l(\Psi_{[\ast]}(\mathscr E^k))$ to faciliate cohomological notation. 
As in \cite[Theorem 4.6]{Alz}, we can now consider the two spectral sequences associated to the first and second filtrations of this bicomplex. Since the $l$-th column of this complex is given by $D^{-l,\bullet}=C_l(\Psi_{[\ast]}(\mathscr E^\bullet))$ and $P$ is finite, the columns of this bicomplex vanish for $l\gg 0$ and hence both spectral sequences converge. Further, by Lemma \ref{L5.2dm}, the rows of this bicomplex are exact up to the $0$-th position, and the cohomologies at the $0$-th position are given by $\Psi(\mathscr E^\bullet)$. It follows that the common abutment of these two spectral sequences is given by the cohomology groups of the complex
\begin{equation}\label{5.7gb}
0\longrightarrow  \Psi(\mathscr E^0)\longrightarrow \Psi(\mathscr E^1) \longrightarrow \dots
\end{equation} Since $0\longrightarrow \mathscr M\longrightarrow \mathscr E^\bullet$ is an injective resolution, it clear that the cohomologies of \eqref{5.7gb} are given by $\mathbb R^\bullet\Psi(\mathscr M)$. Taking the cohomology of the columns of \eqref{DDCC}, one obtains $E_1^{-i,j}=C_i(\mathbb R^j\Psi_{[\ast]}(\mathscr M))$. Now applying the cohomology of the rows of \eqref{DDCC}, we finally obtain the spectral sequence
\begin{equation}
E_2^{-i,j}=\mathbb L_i\text{colim}_{p\in P}\mathbb R^j\Psi_{[\ast]}(\mathscr M)\Longrightarrow \mathbb R^{j-i}\Psi(\mathscr M)
\end{equation} This proves the result.
\end{proof}

In the following three subsections, we will now apply the formalism above to three separate contexts in order to obtain spectral sequences of derived functors. The first of these will be the local cohomology objects in $\mathscr S_R$ with respect to a pair of ideals $(I,J)$.

\subsection{Spectral sequences for $\Gamma_{I,J}:\mathscr S_R\longrightarrow \mathscr S_R$}

We continue with the ideal $I=I_{\alpha_1}\cap ... \cap I_{\alpha_k}$ and the partially ordered set $P$ consisting of all possible sums of the ideals $\{I_{\alpha_1},...,I_{\alpha_k}\}$ as defined above. For an ideal $J\subseteq R$ and $p\in P$, we set $\Psi_p(\_\_):=\Gamma_{I_p,J}(\_\_)$ and $\Psi:=\Gamma_{I,J}$. 

\begin{lem}\label{L5.04xl0}
Let $K$, $J\subseteq R$ be ideals. Then, the functor $\Gamma_{K,J}$ preserves direct sums.
\end{lem}

\begin{proof}
We consider a collection $\{\mathscr M_\beta\}_{\beta\in B}$ of objects in $\mathscr S_R$ and set $\mathscr M:=\underset{\beta\in B}{\bigoplus}\mathscr M_\beta$. From the definition in \eqref{4.1dn}, it is evident that $\Gamma_{K,J}(\mathscr M)\supseteq \underset{\beta\in B}{\bigoplus}\Gamma_{K,J}(\mathscr M_\beta)$. On the other hand, we consider $\mathscr N\in fg(\mathscr M)$ such that $K^n\subseteq Ann(\mathscr N)+J$ for $n\gg 1$. Since $\mathscr N$ is finitely generated, we must have $\mathscr N\subseteq \mathscr M_1 \oplus ...\oplus
\mathscr M_r$ for some finite subcollection $\{\mathscr M_1,...,\mathscr M_r\}$ of objects from $\{\mathscr M_\beta\}_{\beta\in B}$. Let $\rho_l:  \mathscr M_1 \oplus ...\oplus
\mathscr M_r \longrightarrow \mathscr M_l$ denote the canonical projections for $1\leq l\leq r$. We now note that
\begin{equation}
\mbox{$K^n\subseteq Ann(\mathscr N)+J\subseteq Ann(\rho_l(\mathscr N))+J$ for $n\gg 1$}\qquad\Rightarrow \qquad \rho_l(\mathscr N)\subseteq \Gamma_{K,J}(\mathscr M_l)
\end{equation}  for each $1\leq l\leq r$. We now note that $\mathscr N\subseteq \rho_1(\mathscr N)\oplus ...\oplus \rho_r(\mathscr N)\subseteq  \mathscr M_1 \oplus ...\oplus
\mathscr M_r$, which gives $\mathscr N\subseteq  {\bigoplus}_{l=1}^r\Gamma_{K,J}(\mathscr M_l)\subseteq \underset{\beta\in B}{\bigoplus}\Gamma_{K,J}(\mathscr M_\beta)$.
\end{proof}

\begin{lem}\label{L5.4xl}
Let $\mathscr L\in \mathscr S_R$ be an $R$-elementary object and let $\mathfrak p:=Ann(\mathscr L)$. Let $K\subseteq R$ be an ideal. If $\mathfrak{p}\notin \mathbb W(K,J)$, then we have $\Gamma_{K,J}(\mathscr E(\mathscr L))=0$. 
\end{lem}
\begin{proof}
Since $\mathscr L\subseteq\mathscr E(\mathscr L)$ is an essential subobject, it follows by Corollary \ref{C3.7} that  $Ass(\mathscr E(\mathscr L))=Ass(\mathscr L)=\{\mathfrak p\}$. We are given $\mathfrak{p}\notin \mathbb W(K,J)$.  By Proposition \ref{P4.3ed}, we now have $Ass(\Gamma_{K,J}(\mathscr E(\mathscr L)))=Ass(\mathscr E(\mathscr L))\cap \mathbb W(K,J)=\phi$. It follows that $\Gamma_{K,J}(\mathscr E(\mathscr L))=0$. 
\end{proof}

\begin{lem}\label{L5.5hx}
Let $\mathscr L\in \mathscr S_R$ be an $R$-elementary object and let $\mathfrak p:=Ann(\mathscr L)$. Let $K\subseteq R$ be an ideal. 

\smallskip
(a) Suppose that $\mathfrak{p}\in \mathbb W(K,J)$. Then, we have $\Gamma_{K,J}(\mathscr E(\mathscr L))=\mathscr E(\mathscr L)$. 

\smallskip
(b) If $\mathfrak p\not\subseteq \mathfrak m$ for some maximal ideal $\mathfrak m$, then $\mathscr E(\mathscr L)_{\mathfrak m}=0$. 
\end{lem}

\begin{proof}
 (a) Since $\mathfrak p=Ann(\mathscr L)$ and $\mathfrak{p}\in \mathbb W(K,J)$, we have $K^n\subseteq Ann(\mathscr L)+J$ for $n\gg 1$. It follows from the definition in \eqref{4.1dn} that $\mathscr L$ is $(K,J)$-torsion. By Lemma \ref{L4.7x4}, it follows that the injective hull $\mathscr E(\mathscr L)$ is also $(K,J)$-torsion, i.e., $\Gamma_{K,J}(\mathscr E(\mathscr L))=\mathscr E(\mathscr L)$.  
 
 \smallskip
(b)  Since $\mathscr L$ is finitely generated, we know that $Supp(\mathscr L)=\mathbb V(Ann(\mathscr L))=\mathbb V(\mathfrak p)$. Accordingly, if $\mathfrak p\not\subseteq \mathfrak m$, then $\mathscr L_{\mathfrak m}=0$. By Proposition \ref{P2.14ju}, we know that $\mathscr E(\mathscr L)_{\mathfrak m}$ is the injective hull of $\mathscr L_{\mathfrak m}$, and hence $\mathscr E(\mathscr L)_{\mathfrak m}=0$. 
\end{proof}

\begin{thm}\label{P5.6fq}  For any $\mathscr M\in \mathscr S_R$, we have a spectral sequence
\begin{equation}
E_2^{-i,j}=\mathbb L_i\text{colim}_{p\in P}  H^j_{[\ast],J}(\mathscr M)\Longrightarrow H^{j-i}_{I,J}(\mathscr M)
\end{equation}

\end{thm}

\begin{proof}
This follows directly from Lemmas \ref{L5.04xl0}, \ref{L5.4xl},   \ref{L5.5hx} and Theorem \ref{T5.3cr}. 
\end{proof}

\subsection{Generalized local cohomology objects on $\mathscr S_R$}

  Let $V$ be an $R$-module. By \cite[$\S$ B4]{AZ}, we know that $\otimes_RV:\mathscr S_R\longrightarrow \mathscr S_R$ has a right adjoint $\underline{Hom}_R(V,\_\_)$. The bifunctor 
$\underline{Hom}_R(\_\_,\_\_):(R-Mod)^{op}\times \mathscr S_R\longrightarrow \mathscr S_R$ is left exact in both variables and satisfies
$\underline{Hom}_R(R,\mathscr M)=\mathscr M$ for any $\mathscr M\in \mathscr S_R$. The right derived functors of $\underline{Hom}_R(V,\_\_)$ are denoted by
$\underline{Ext}^\bullet_R(V,\_\_)$.  If $V$, $W$ are $R$-modules, it also follows from the adjunction that 
$\underline{Hom}_R(V\otimes_RW,\_\_)\cong \underline{Hom}_R(V,\underline{Hom}_R(W,\_\_))$. 

\smallskip
Given an $R$-module $V$ and any ideal $K\subseteq R$, we denote by $V_K$ the inverse system given by $V_K:=\{V/K^tV\}_{t\geq 1}$. We now set
\begin{equation}\label{cr510g}
\Gamma_{V_K}:\mathscr S_R \longrightarrow \mathscr S_R \qquad \mathscr M\mapsto \underset{t\geq 1}{\varinjlim}\textrm{ }\underline{Hom}_R(V/K^tV,\mathscr M)
\end{equation} It is immediate from \eqref{cr510g} that $\Gamma_{V_K}$ is left exact and we denote by $H^\bullet_{V_K}(\_\_)$ the right derived functors
of $\Gamma_{V_K}$. 

 \begin{thm}\label{crP4.3} Let $V$ be an $R$-module and let $K\subseteq R$ be an ideal. Then, for any $\mathscr M\in \mathscr S_R$, we have 
 \begin{equation}
 H^i_{V_K}(\mathscr M)=\underset{t\geq 1}{\varinjlim}\textrm{ }\underline{Ext}^i_R(V/K^tV,\mathscr M) \qquad i\geq 0
 \end{equation}
 \end{thm}
 \begin{proof}
 For the sake of convenience, we set $F^i(\mathscr M):=\underset{t\geq 1}{\varinjlim}\textrm{ }\underline{Ext}^i_R(V/K^tV,\mathscr M)$ for any $\mathscr M\in
 \mathscr S_R$. From  the defintion in 
 \ref{cr510g}, it follows that $H^0_{V_K}(\mathscr M)=F^0(\mathscr M)$. The derived functors $\{H^\bullet_{V_K}(\_\_)\}$ give a family of cohomological $\delta$-functors that is universal. Similarly, for each $t\geq 1$, the derived functors $\{\underline{Ext}^\bullet_R(V/K^tV,\_\_)\}$ are universal $\delta$-functors on $\mathscr S_R$. Since filtered colimits in 
 $\mathscr S_R$ are also exact, we see that $\{F^\bullet\}$ is also a family of $\delta$-functors. For any injective $\mathscr E\in \mathscr S_R$, the derived functor $\underline{Ext}^i_R(V/K^tV,\mathscr E)$ vanishes for $i>0$ and hence so does $F^i(\mathscr E)$. It follows that $\{F^\bullet\}$ is a universal $\delta$-functor and since $F^0=H^0_{V_K}$, we must have
 $F^i=H^i_{V_K}$ for every $i$. 
 \end{proof}

In order to understand the functor $\Gamma_{V_K}$ better, we define, for any ideal $K\subseteq R$ and 
$\mathscr M\in \mathscr S_R$
\begin{equation}\label{creq4.1}
\gamma_K(\mathscr M):=\underset{a\in K}{\bigcap}\textrm{ }Ker(a_{\mathscr M}:\mathscr M\longrightarrow \mathscr M)=\underset{i=1}{\overset{n}{\bigcap}}\textrm{ }Ker(a_{i,\mathscr M}:\mathscr M\longrightarrow \mathscr M)
\end{equation} where $\{a_1,...,a_n\}$ is a set of generators for $K$. 

\begin{lem}\label{crL4.1}
Let $K\subseteq R$ be an ideal and $\mathscr M\in \mathscr S_R$. Then, $\gamma_K(\mathscr M)=\underline{Hom}_R(R/K,\mathscr M)$. 
\end{lem}

\begin{proof} If $\{a_1,...,a_n\}$ is a set of generators for $K$, we note that $R/K=(R/a_1R)\otimes_R ...\otimes_R(R/a_nR)$ and hence
$\underline{Hom}_R(R/K,\mathscr M)=\underline{Hom}_R(R/a_1R,\underline{Hom}_R(R/a_2R,...\underline{Hom}_R(R/a_nR,\mathscr M)))$. Therefore, it suffices to check the result for a principal ideal $K=(a)$. In that case, $R\overset{a}{\longrightarrow}R\longrightarrow R/aR\longrightarrow 0$ is a free presentation of $R/aR$ and it follows from \cite[(B4.2)]{AZ} that
\begin{equation}
0\longrightarrow \underline{Hom}_R(R/aR,\mathscr M)\longrightarrow \underline{Hom}_R(R,\mathscr M)\cong \mathscr M \xrightarrow{\underline{Hom}_R(R,a_{\mathscr M})=a_{\mathscr M}}\underline{Hom}_R(R,\mathscr M)\cong \mathscr M
\end{equation} is exact. The result is now clear.

\end{proof}

From Lemma \ref{crL4.1} and the definition in \eqref{4.1dn}, it now follows that for any ideal $K\subseteq R$ we have
 \begin{equation}\label{crgam}
\Gamma_K(\mathscr M):=\Gamma_{K,0}(\mathscr M)=\underset{t\geq 1}{\bigcup}\textrm{ }\gamma_{K^t}(\mathscr M)= \underset{t\geq 1}{\varinjlim}\textrm{ }\underline{Hom}_R(R/K^t,\mathscr M)
 \end{equation} For the rest of this subsection, we suppose that $V$ is a finitely generated $R$-module. We note here the following fact.

\begin{lem}\label{L5.9ut} Let $V$ be a finitely generated $R$-module. Then, the functor $\underline{Hom}_R(V,\_\_):\mathscr S_R
\longrightarrow \mathscr S_R$ preserves filtered colimits.
\end{lem}
 
\begin{proof}
Let $\{\mathscr M_\beta\}_{\beta\in B}$ be a filtered system of objects in $\mathscr S_R$. We will show that
\begin{equation}\label{514gv}\mathscr S_R\left(\mathscr N\otimes_RV,\underset{\beta\in B}{\varinjlim}\textrm{ } \mathscr M_\beta\right)=\mathscr S_R\left(\mathscr N,\underline{Hom}_R\left(V,\underset{\beta\in B}{\varinjlim}\textrm{ } \mathscr M_\beta\right)\right)=\mathscr S_R\left(\mathscr N,
\underset{\beta\in B}{\varinjlim}\textrm{ }\underline{Hom}_R(V,\mathscr M_\beta)\right)
\end{equation} for any $\mathscr N\in \mathscr S_R$. Since $\mathscr S_R$ is locally noetherian, any object of $\mathscr S_R$ may be expressed as a filtered colimit of its finitely generated subobjects. Accordingly, it suffices to check \eqref{514gv} for $\mathscr N\in \mathscr S_R$ finitely generated. If $\mathscr N$ is  finitely generated, we have 
\begin{equation}\label{515gv}
\mathscr S_R\left(\mathscr N,
\underset{\beta\in B}{\varinjlim}\textrm{ }\underline{Hom}_R(V,\mathscr M_\beta)\right)=\underset{\beta\in B}{\varinjlim}\textrm{ } \mathscr S_R(\mathscr N,\underline{Hom}_R(V,\mathscr M_\beta))= \underset{\beta\in B}{\varinjlim}\textrm{ }\mathscr S_R(\mathscr N\otimes_RV,\mathscr M_\beta)
\end{equation} From \eqref{514gv} and \eqref{515gv}, we see that it suffices to check 
\begin{equation}\label{517gv}
\mathscr S_R\left(\mathscr N\otimes_RV,\underset{\beta\in B}{\varinjlim}\textrm{ } \mathscr M_\beta\right)=\underset{\beta\in B}{\varinjlim}\textrm{ }\mathscr S_R(\mathscr N\otimes_RV,\mathscr M_\beta)
\end{equation} for $\mathscr N\in \mathscr S_R$ finitely generated. But we are given that $V$ is a finitely generated $R$-module, i.e., a quotient of $R^n$ for some $n\geq 1$. Then, $\mathscr N\otimes_RV$ is a quotient of $\mathscr N^n$ and hence $\mathscr N\otimes_RV$ is finitely generated in $\mathscr S_R$. The result is now clear.
\end{proof} 

\begin{thm}\label{P5.10pyz} Let $V$ be a finitely generated $R$-module and $K\subseteq R$ an ideal. Then, for any $\mathscr M\in \mathscr S_R$, we have
$\Gamma_{V_K}(\mathscr M)=\underline{Hom}_R(V,\Gamma_K(\mathscr M))$. In particular, $\Gamma_{V_K}$ preserves direct sums. 
\end{thm}
\begin{proof}
Since $V$ is finitely generated,  it follows from \eqref{cr510g}, \eqref{crgam} and Lemma \ref{L5.9ut} that
\begin{equation*}
\begin{array}{ll}
\Gamma_{V_K}(\mathscr M)= \underset{t\geq 1}{\varinjlim}\textrm{ }\underline{Hom}_R(V/K^tV,\mathscr M)=\underset{t\geq 1}{\varinjlim}\textrm{ }\underline{Hom}_R(V,\underline{Hom}_R(R/K^t,\mathscr M))&=\underline{Hom}_R(V,\underset{t\geq 1}{\varinjlim}\textrm{ }\underline{Hom}_R(R/K^t,\mathscr M))\\
&=\underline{Hom}_R(V,\Gamma_K(\mathscr M))\\
\end{array}
\end{equation*} The last statement is clear from Lemma \ref{L5.9ut} and the fact that $\Gamma_K=\Gamma_{K,0}$ preserves direct sums.
\end{proof}

\begin{lem}\label{L5115mcp}
Let $R\longrightarrow R'$ be an extension of $k$-algebras. For any $R$-module $W$ and any $\mathscr N\in \mathscr S_{R'}$, we have an isomorphism
\begin{equation}
\underline{Hom}_{R'}(W\otimes_RR',\mathscr N)\cong \underline{Hom}_R(W,\mathscr N)
\end{equation}
\end{lem}

\begin{proof}
We write $W$ as the cokernel $R^{(Y)}\longrightarrow R^{(X)}\longrightarrow W\longrightarrow 0$ of free modules. Then, $W\otimes_RR'$ is expressed as the cokernel $R'^{(Y)}\longrightarrow R'^{(X)}\longrightarrow W\otimes_RR'\longrightarrow 0$. By definition (see \cite[B4.2]{AZ}), $\underline{Hom}_R(W,\mathscr N)$ is given as the kernel of the induced map $\mathscr N^{X}=\underline{Hom}_R(R^{(X)},\mathscr N)\longrightarrow 
\underline{Hom}_R(R^{(Y)},\mathscr N)=\mathscr N^{Y}$.  Similarly, $\underline{Hom}_{R'}(W\otimes_RR',\mathscr N)$ is also given by the kernel
$\mathscr N^{X}=\underline{Hom}_{R'}(R'^{(X)},\mathscr N)\longrightarrow 
\underline{Hom}_{R'}(R'^{(Y)},\mathscr N)=\mathscr N^{Y}$. 
\end{proof}

\begin{lem}\label{L511pcm} Let $R\longrightarrow R[T^{-1}]$ be the localization of $R$ with respect to a multiplicatively closed subset $T$. Then, for any finitely generated $R$-module $V$ and any $\mathscr M\in \mathscr S_R$, we have
\begin{equation}
\underline{Hom}_R(V,\mathscr M)_T\cong \underline{Hom}_{R[T^{-1}]}(V_T,\mathscr M_T)
\end{equation}

\end{lem}
 \begin{proof}
Applying Lemma \ref{L5115mcp}, we know that $ \underline{Hom}_{R[T^{-1}]}(V_T,\mathscr M_T)\cong \underline{Hom}_R(V,\mathscr M_T)$. We have noted before that the localization $\mathscr M_T$ is given by a filtered colimit. Since $V$ is a  finitely generated $R$-module, we know from Lemma \ref{L5.9ut} that
$\underline{Hom}_R(V,\_\_)$ preserves filtered colimits. We now have $ \underline{Hom}_R(V,\mathscr M_T)\cong \underline{Hom}_R(V,\mathscr M)_T$.
\end{proof}

We now return to the  ideal $I=I_{\alpha_1}\cap ... \cap I_{\alpha_k}$ and the partially ordered set $P$ consisting of all possible sums of the ideals $\{I_{\alpha_1},...,I_{\alpha_k}\}$ as  before. We fix $J=0$ and a finitely generated $R$-module $V$. For $p\in P$, we set $\Psi_p(\_\_):=\Gamma_{V_{I_p}}(\_\_)$ and $\Psi:=\Gamma_{V_I}$. 

\begin{lem}\label{513lemy}
Let $\mathscr L\in \mathscr S_R$ be an $R$-elementary object and let $\mathfrak p:=Ann(\mathscr L)$. Let $K\subseteq R$ be an ideal. 

\smallskip
(a) If $\mathfrak{p}\notin \mathbb W(K,0)=\mathbb V(K)$, then we have $\Gamma_{V_K}(\mathscr E(\mathscr L))=0$. 

\smallskip
(b) Suppose that $\mathfrak{p}\in \mathbb W(K,0)=\mathbb V(K)$. Then we have $\Gamma_{V_K}(\mathscr E(\mathscr L))=\underline{Hom}_R(V,\mathscr E(\mathscr L))$.  
\end{lem}

\begin{proof} (a) 
Since $\mathfrak{p}\notin \mathbb W(K,0)=\mathbb V(K)$,  it follows by Lemma \ref{L5.4xl} that $\Gamma_{K}(\mathscr E(\mathscr L))=\Gamma_{K,0}(\mathscr E(\mathscr L))=0$.  Since $V$ is finitely generated, we have  by Proposition \ref{P5.10pyz} that  $\Gamma_{V_K}(\mathscr E(\mathscr L))=\underline{Hom}_R(V,\Gamma_K(\mathscr E(\mathscr L)))=0$.

\smallskip
(b) Since  $\mathfrak{p}\in \mathbb W(K,0)=\mathbb V(K)$, it follows by Lemma \ref{L5.5hx} that $\Gamma_{K}(\mathscr E(\mathscr L))=\Gamma_{K,0}(\mathscr E(\mathscr L))=\mathscr E(\mathscr L)$.  Since $V$ is finitely generated, we have  by Proposition \ref{P5.10pyz} that  $\Gamma_{V_K}(\mathscr E(\mathscr L))=\underline{Hom}_R(V,\Gamma_K(\mathscr E(\mathscr L)))=\underline{Hom}_R(V,\mathscr E(\mathscr L))$. 
\end{proof}

\begin{lem}\label{514lemy}
Let $\mathscr L\in \mathscr S_R$ be an $R$-elementary object and let $\mathfrak p:=Ann(\mathscr L)$. Let $K\subseteq R$ be an ideal. 
 If $\mathfrak p\not\subseteq \mathfrak m$ for some maximal ideal $\mathfrak m$, then $\Gamma_{V_K}(\mathscr E(\mathscr L))_{\mathfrak m}=0$. 
\end{lem}

\begin{proof} 
If   $\mathfrak{p}\notin \mathbb W(K,0)=\mathbb V(K)$,  we already  have $\Gamma_{V_K}(\mathscr E(\mathscr L))=0$ by Lemma \ref{513lemy}(a). Otherwise, suppose $\mathfrak{p}\in \mathbb W(K,0)=\mathbb V(K)$. Then Lemma \ref{513lemy}(b) gives us $\Gamma_{V_K}(\mathscr E(\mathscr L))=\underline{Hom}_R(V,\mathscr E(\mathscr L))$.   Because $V$ is finitely generated, applying Lemma \ref{L511pcm} gives us $\Gamma_{V_K}(\mathscr E(\mathscr L))_{\mathfrak m}=\underline{Hom}_{R_{\mathfrak m}}(V_{\mathfrak m},\mathscr E(\mathscr L)_{\mathfrak m})$ for any maximal ideal $\mathfrak m$. Since $\mathfrak p\not\subseteq \mathfrak m$, we have by Lemma \ref{L5.5hx} that $\mathscr E(\mathscr L)_{\mathfrak m}=0$, which shows that $\Gamma_{V_K}(\mathscr E(\mathscr L))_{\mathfrak m}=0$.

\end{proof}

\begin{thm}\label{515mth} Let $ V$ be a finitely generated $R$-module. Then, for any $\mathscr M\in \mathscr S_R$, we have a spectral sequence
\begin{equation}
E_2^{-i,j}=\mathbb L_i\text{colim}_{p\in P} H^j_{V_{I_p}}(\mathscr M)\Longrightarrow H^{j-i}_{V_I}(\mathscr M)
\end{equation}

\end{thm}

\begin{proof}
This follows directly from Lemma \ref{513lemy}, Lemma \ref{514lemy} and Theorem \ref{T5.3cr}. 
\end{proof}

\subsection{Generalized Nagata ideal transforms on $\mathscr S_R$}

We let $V$ be a finitely generated $R$-module and $K\subseteq R$ be an ideal.  We define the generalized Nagata ideal transform $\Delta_{V_K}$ on $\mathscr S_R$ as follows
\begin{equation}\label{521gnt}
\Delta_{V_K}:\mathscr S_R\longrightarrow \mathscr S_R\qquad \mathscr M\mapsto \underset{t\geq 1}{\varinjlim}\textrm{ }\underline{Hom}_R(K^tV,\mathscr M)
\end{equation} It is immediate that the functor $\Delta_{V_K}$ is left exact. Since $V$ is finitely generated, it is also clear from Lemma \ref{L5.9ut} that $\Delta_{V_K}$ preserves direct sums. We now need the following result.

\begin{lem}\label{520mely}
Let $\mathscr E\in \mathscr S_R$ be an injective object. Then, the functor $\underline{Hom}_R(\_\_,\mathscr E):R-Mod\longrightarrow \mathscr S_R$ is exact.
\end{lem}

\begin{proof}
We already know that $\underline{Hom}_R(\_\_,\mathscr E)$ is left exact. Let $V'\hookrightarrow V$ be an inclusion of $R$-modules. We will show that $\underline{Hom}_R(V,\mathscr E)
\longrightarrow \underline{Hom}_R(V',\mathscr E)$ is an epimorphism. Accordingly, we set $\mathscr M\in \mathscr S_R$ to be the cokernel $\underline{Hom}_R(V,\mathscr E)
\longrightarrow \underline{Hom}_R(V',\mathscr E)\longrightarrow \mathscr M\longrightarrow 0$. 

\smallskip In this section, we have assumed that $\mathscr S$ has enough projectives. We choose an epimorphism $\mathscr P\twoheadrightarrow\mathscr M$ in $\mathscr S$ with $\mathscr P\in \mathscr S$ projective. This induces an epimorphism $\mathscr P\otimes R\twoheadrightarrow\mathscr M\otimes R$ in $\mathscr S_R$. Also, we know that composing the structure map $\mathscr M\otimes R\longrightarrow \mathscr M$ of $\mathscr M\in \mathscr S_R$ with the morphism  $\mathscr M=\mathscr M\otimes k
\longrightarrow \mathscr M\otimes R$ induced by the unit gives the identity in $\mathscr S$.  Accordingly,
$\mathscr M\otimes R\longrightarrow \mathscr M$  is an epimorphism in $\mathscr S_R$, since the underlying morphism in $\mathscr S$ is an epimorphism. We now have an epimorphism
$\mathscr P\otimes R\twoheadrightarrow\mathscr M\otimes R\twoheadrightarrow \mathscr M$ in $\mathscr S_R$. 

\smallskip Since $\mathscr P\in \mathscr S$ is projective, the adjoint isomorphism $\mathscr S_R(\mathscr P\otimes R,\_\_)\cong \mathscr S_k(\mathscr P,\_\_)$ shows that $\mathscr P\otimes R\in \mathscr S_R$ is projective. Hence, the following sequence is exact
\begin{equation}\label{523cap}
\mathscr S_R(\mathscr P\otimes R,\underline{Hom}_R(V,\mathscr E))
\longrightarrow \mathscr S_R(\mathscr P\otimes R,\underline{Hom}_R(V',\mathscr E))\longrightarrow \mathscr S_R(\mathscr P\otimes R, \mathscr M)\longrightarrow 0
\end{equation} Since $k$ is a field, the inclusion $V'\hookrightarrow V$ induces a monomorphism  $(\mathscr P\otimes R)\otimes_RV'=\mathscr P\otimes V'\longrightarrow 
\mathscr P\otimes V=(\mathscr P\otimes R)\otimes_RV$. Since $\mathscr E\in \mathscr S_R$ is injective, it now follows that $\mathscr S_R((\mathscr P\otimes R)\otimes_RV,
\mathscr E)\longrightarrow \mathscr S_R((\mathscr P\otimes R)\otimes_RV',
\mathscr E) $ is an epimorphism. Using \eqref{523cap}, we now have $ \mathscr S_R(\mathscr P\otimes R, \mathscr M)=0$, which shows that $\mathscr M=0$. 

\end{proof}

Using Lemma \ref{520mely}, we see that  if $\mathscr E\in \mathscr S_R$ is an injective object, then 
$0\longrightarrow \underline{Hom}_R(V/K^tV,\mathscr E)\longrightarrow \underline{Hom}_R(V,\mathscr E)
\longrightarrow \underline{Hom}_R(K^tV,\mathscr E)\longrightarrow 0$ is exact for any $t\geq 1$. Taking filtered colimits, we have a short exact sequence
\begin{equation}\label{gnte}
0\longrightarrow  \Gamma_{V_K}(\mathscr E)=\underset{t\geq 1}{\varinjlim}\textrm{ }\underline{Hom}_R(V/K^tV,\mathscr E)\longrightarrow \underline{Hom}_R(V,\mathscr E)
\longrightarrow \underset{t\geq 1}{\varinjlim}\textrm{ } \underline{Hom}_R(K^tV,\mathscr E)=\Delta_{V_K}(\mathscr E)\longrightarrow 0
\end{equation} For our setup with $I=I_{\alpha_1}\cap ... \cap I_{\alpha_k}$ and $J=0$, we now set $\Psi_p:=\Delta_{V_{I_p}}(\_\_)$ 
for each $p\in P$ and $\Psi:=\Delta_{V_I}$.

\begin{lem}\label{516lemy}
Let $\mathscr L\in \mathscr S_R$ be an $R$-elementary object and let $\mathfrak p:=Ann(\mathscr L)$. Let $K\subseteq R$ be an ideal.

\smallskip
(a) If $\mathfrak{p}\in \mathbb W(K,0)=\mathbb V(K)$, then we have $\Delta_{V_K}(\mathscr E(\mathscr L))=0$. 

\smallskip
(b)  If $\mathfrak{p}\notin \mathbb W(K,0)=\mathbb V(K)$, we have $\Delta_{V_K}(\mathscr E(\mathscr L))=\underline{Hom}_R(V,\mathscr E(\mathscr L))$. 
\end{lem}
\begin{proof} (a)
If $\mathfrak{p}\in \mathbb W(K,0)=\mathbb V(K)$, it follows by Lemma \ref{513lemy}(b) that  $\Gamma_{V_K}(\mathscr E(\mathscr L))=\underline{Hom}_R(V,\mathscr E(\mathscr L))$.   Accordingly, the short exact sequence in \eqref{gnte} gives us  $\Delta_{V_K}(\mathscr E(\mathscr L))=0$.  

\smallskip
(b) If $\mathfrak{p}\notin \mathbb W(K,0)=\mathbb V(K)$, it follows by Lemma \ref{513lemy}(a) that  $\Gamma_{V_K}(\mathscr E(\mathscr L))=0$.   Accordingly, the short exact sequence in \eqref{gnte} gives us  $\Delta_{V_K}(\mathscr E(\mathscr L))=\underline{Hom}_R(V,\mathscr E(\mathscr L))$.  
\end{proof}
 
\begin{lem}\label{517lemy}
Let $\mathscr L\in \mathscr S_R$ be an $R$-elementary object and let $\mathfrak p:=Ann(\mathscr L)$. Let $K\subseteq R$ be an ideal. If $\mathfrak p\not\subseteq \mathfrak m$ for some maximal ideal $\mathfrak m$, then $\Delta_{V_K}(\mathscr E(\mathscr L))_{\mathfrak m}=0$. 
\end{lem}
\begin{proof} If $\mathfrak{p}\in \mathbb W(K,0)$, we already know from Lemma \ref{516lemy}(a) that $\Delta_{V_K}(\mathscr E(\mathscr L))=0$.  Otherwise, suppose that $\mathfrak{p}\notin \mathbb W(K,0)$, so that $\Delta_{V_K}(\mathscr E(\mathscr L))=\underline{Hom}_R(V,\mathscr E(\mathscr L))$ by Lemma \ref{516lemy}(b). 
Because $V$ is finitely generated, applying Lemma \ref{L511pcm} gives us $\Delta_{V_K}(\mathscr E(\mathscr L))_{\mathfrak m}=\underline{Hom}_{R_{\mathfrak m}}(V_{\mathfrak m},\mathscr E(\mathscr L)_{\mathfrak m})$ for any maximal ideal $\mathfrak m$.  Since $\mathfrak p\not\subseteq \mathfrak m$, we have by Lemma \ref{L5.5hx}(b) that $\mathscr E(\mathscr L)_{\mathfrak m}=0$, which shows that $\Delta_{V_K}(\mathscr E(\mathscr L))_{\mathfrak m}=0$.

\end{proof}

\begin{thm}\label{515mth0}  Let $ V$ be a finitely generated $R$-module.  For any $\mathscr M\in \mathscr S_R$, we have a spectral sequence
\begin{equation}
E_2^{-i,j}=\mathbb L_i\text{colim}_{p\in P}\mathbb R^j\Delta_{V_{I_p}}(\mathscr M)\Longrightarrow \mathbb R^{j-i}\Delta_{V_I}(\mathscr M)
\end{equation}

\end{thm}

\begin{proof}
This follows directly from Lemma \ref{516lemy}, Lemma \ref{517lemy} and Theorem \ref{T5.3cr}. 
\end{proof}


\medskip

\medskip

\medskip




\small

\begin{bibdiv}
	\begin{biblist}
	
	\bib{AR}{book}{
   author={Ad\'{a}mek, J.},
   author={Rosick\'{y}, J.},
   title={Locally presentable and accessible categories},
   series={London Mathematical Society Lecture Note Series},
   volume={189},
   publisher={Cambridge University Press, Cambridge},
   date={1994},
   pages={xiv+316},
}

\bib{Alz2}{article}{
   author={\`Alvarez Montaner, J.},
   author={Garc\'{\i}a L\'{o}pez, R.},
   author={Zarzuela Armengou, S.},
   title={Local cohomology, arrangements of subspaces and monomial ideals},
   journal={Adv. Math.},
   volume={174},
   date={2003},
   number={1},
   pages={35--56},
}
	
	
	\bib{Alz}{article}{
   author={\`Alvarez Montaner, J.},
   author={Boix, A. F.},
   author={Zarzuela, S.},
   title={On some local cohomology spectral sequences},
   journal={Int. Math. Res. Not. IMRN},
   date={2020},
   number={19},
   pages={6197--6293},
}
	

\bib{AZ0}{article}{
   author={Artin, M.},
   author={Zhang, J. J.},
   title={Noncommutative projective schemes},
   journal={Adv. Math.},
   volume={109},
   date={1994},
   number={2},
   pages={228--287},
}

\bib{ASZ}{article}{
   author={Artin, M.},
   author={Small, L. W.},
   author={Zhang, J. J.},
   title={Generic flatness for strongly Noetherian algebras},
   journal={J. Algebra},
   volume={221},
   date={1999},
   number={2},
   pages={579--610},

}
	
	\bib{AZ}{article}{
   author={Artin, M.},
   author={Zhang, J. J.},
   title={Abstract Hilbert schemes},
   journal={Algebr. Represent. Theory},
   volume={4},
   date={2001},
   number={4},
   pages={305--394},
}
	
	
	\bib{Banj}{article}{
   author={Banerjee, A.},
   title={An extension of the Beauville-Laszlo descent theorem},
   journal={Arch Math (Basel)},
   date={to appear},
pages={DOI: 10.1007/s00013-023-01854-1},
}
	
	\bib{BeRe}{article}{
   author={Beligiannis, A.},
   author={Reiten, I.},
   title={Homological and homotopical aspects of torsion theories},
   journal={Mem. Amer. Math. Soc.},
   volume={188},
   date={2007},
   number={883},
   pages={viii+207},

}

\bib{Brd79}{article}{
   author={Brodmann, M.},
   title={Asymptotic stability of ${\rm Ass}(M/I\sp{n}M)$},
   journal={Proc. Amer. Math. Soc.},
   volume={74},
   date={1979},
   number={1},
   pages={16--18},
}

\bib{BHell}{article}{
   author={Brodmann, M.},
   author={Hellus, M.},
   title={Cohomological patterns of coherent sheaves over projective
   schemes},
   journal={J. Pure Appl. Algebra},
   volume={172},
   date={2002},
   number={2-3},
   pages={165--182},
}
	


	\bib{BLash}{article}{
   author={Brodmann, M. P.},
   author={Faghani, A. Lashgari},
   title={A finiteness result for associated primes of local cohomology
   modules},
   journal={Proc. Amer. Math. Soc.},
   volume={128},
   date={2000},
   number={10},
   pages={2851--2853},
}



\bib{BRS}{article}{
   author={Brodmann, M.},
   author={Rotthaus, Ch.},
   author={Sharp, R. Y.},
   title={On annihilators and associated primes of local cohomology modules},
   journal={J. Pure Appl. Algebra},
   volume={153},
   date={2000},
   number={3},
   pages={197--227},
}
	
\bib{BS}{book}{
   author={Brodmann, M. P.},
   author={Sharp, R. Y.},
   title={Local cohomology},
   series={Cambridge Studies in Advanced Mathematics},
   volume={136},
   edition={2},
   publisher={Cambridge University Press, Cambridge},
   date={2013},
   pages={xxii+491},
}


\bib{Dasc}{book}{
   author={D\u{a}sc\u{a}lescu, S.},
   author={N\u{a}st\u{a}sescu, C.},
   author={Raianu, \c{S}.},
   title={Hopf algebras},
   series={Monographs and Textbooks in Pure and Applied Mathematics},
   volume={235},
   note={An introduction},
   publisher={Marcel Dekker, Inc., New York},
   date={2001},
}

\bib{DAS}{article}{
   author={Divaani-Aazar, K.},
   author={Sazeedeh, R.},
   title={Cofiniteness of generalized local cohomology modules},
   journal={Colloq. Math.},
   volume={99},
   date={2004},
   number={2},
   pages={283--290},
}

\bib{Falt}{article}{
   author={Faltings, G.},
   title={Der Endlichkeitssatz in der lokalen Kohomologie},
   journal={Math. Ann.},
   volume={255},
   date={1981},
   number={1},
   pages={45--56},
}
	
	
	\bib{Hell}{article}{
   author={Hellus, M.},
   title={On the set of associated primes of a local cohomology module},
   journal={J. Algebra},
   volume={237},
   date={2001},
   number={1},
   pages={406--419},
}

\bib{Hun}{article}{
   author={Huneke, C},
   author={Katz, D.},
   author={Marley, T.},
   title={On the support of local cohomology},
   journal={J. Algebra},
   volume={322},
   date={2009},
   number={9},
   pages={3194--3211},
}
	
\bib{LV}{article}{
   author={Lowen, W.},
   author={Van den Bergh, M.},
   title={Deformation theory of abelian categories},
   journal={Trans. Amer. Math. Soc.},
   volume={358},
   date={2006},
   number={12},
   pages={5441--5483},
}


\bib{Neem}{book}{
   author={Neeman, A.},
   title={Triangulated categories},
   series={Annals of Mathematics Studies},
   volume={148},
   publisher={Princeton University Press, Princeton, NJ},
   date={2001},
   pages={viii+449},
}


\bib{Pop}{book}{
   author={Popescu, N.},
   title={Abelian categories with applications to rings and modules},
   series={London Mathematical Society Monographs, No. 3},
   publisher={Academic Press, London-New York},
   date={1973},
   pages={xii+467},
}

\bib{RS}{article}{
   author={Rotthaus, C.},
   author={\c{S}ega, L. M.},
   title={Some properties of graded local cohomology modules},
   journal={J. Algebra},
   volume={283},
   date={2005},
   number={1},
   pages={232--247},
}

\bib{Sten}{book}{
   author={Stenstr\"{o}m, Bo},
   title={Rings of quotients},
   series={Die Grundlehren der mathematischen Wissenschaften, Band 217},
   note={An introduction to methods of ring theory},
   publisher={Springer-Verlag, New York-Heidelberg},
   date={1975},
   pages={viii+309},
}

\bib{TYY}{article}{
   author={Takahashi, R.},
   author={Yoshino, Y.},
author={Yoshizawa, T.},
   title={Local cohomology based on a nonclosed support defined by a pair of ideals},
   journal={J. Pure Appl. Algebra},
   volume={213},
   date={2009},
   number={4},
   pages={582--600},
}
	
	
	\bib{Web}{book}{
   author={Weibel, C.~A.},
   title={An introduction to homological algebra},
   series={Cambridge Studies in Advanced Mathematics},
   volume={38},
   publisher={Cambridge University Press, Cambridge},
   date={1994},
   pages={xiv+450},
}
	


	
	\end{biblist}
	
	\end{bibdiv}
\end{document}